\documentclass{amsart}
\usepackage{amsmath, amsthm, amssymb, amsfonts}
\usepackage[normalem]{ulem}
\usepackage{hyperref}

\usepackage{xypic, mathtools}
\usepackage{verbatim} 
\usepackage{longtable} 

\DeclarePairedDelimiter\floor{\lfloor}{\rfloor}

\theoremstyle{plain}
\newtheorem{thm}{Theorem}
\newtheorem{cor}{Corollary}
\newtheorem{lemma}{Lemma}
\newtheorem{prop}{Proposition}
\newtheorem{conj}{Conjecture}

\theoremstyle{definition}

\theoremstyle{remark}

\usepackage{tikz}
\tikzstyle{vertex}=[circle, draw, inner sep=0pt, minimum size=6pt, fill]
\newcommand{\vertex}{\node[vertex]}

\usepackage{lmodern}
\usetikzlibrary{decorations.pathmorphing}
\tikzset{snake it/.style={decorate, decoration=snake}}

\newcommand{\BC}{{\mathbb{C}}}

\newcommand{\BE}{{\mathbb{E}}}

\newcommand{\BH}{{\mathbb{H}}}

\newcommand{\BQ}{{\mathbb{Q}}}
\newcommand{\BR}{{\mathbb{R}}}

\newcommand{\BZ}{{\mathbb{Z}}}

\newcommand{\CF}{{\mathcal F}}

\newcommand{\CO}{{\mathcal O}}
\newcommand{\CP}{{\mathcal P}}

\newcommand{\Fp}{{\mathfrak{p}}}

\newcommand{\Fz}{{\mathfrak{z}}}

\newcommand{\blangle}{\big\langle}
\newcommand{\brangle}{\big\rangle}
\newcommand{\Blangle}{\Big\langle}
\newcommand{\Brangle}{\Big\rangle}

\newcommand{\QMod}{\mathrm{QMod}}
\newcommand{\Mod}{\mathrm{Mod}}

\DeclareMathOperator{\Hilb}{Hilb}
\DeclareMathOperator{\Aut}{Aut}

\DeclareMathOperator{\id}{id}

\newcommand{\Pic}{\mathop{\rm Pic}\nolimits}

\newcommand{\Mbar}{{\overline M}}

\newcommand{\ev}{\mathop{\rm ev}\nolimits}
\newcommand{\p}{\mathbb{P}}

\newcommand{\e}{\mathbf{1}}

\newcommand{\pt}{{\mathsf{p}}}

\newcommand{\vacuum}{{v_{\varnothing}}}
\newcommand{\hodge}{{\BE^{\vee}(1)}}

\newcommand{\QJac}{\mathop{\rm QJac}\nolimits}

\usepackage{ifthen}
\newboolean{isExtended}
\setboolean{isExtended}{false} 
\long\def\Extended#1{ \ifthenelse{\boolean{isExtended}}{  \noindent {\bf Extended:}\ #1 }{} }

\DeclareFontFamily{OT1}{rsfs}{}
\DeclareFontShape{OT1}{rsfs}{n}{it}{<-> rsfs10}{}
\DeclareMathAlphabet{\curly}{OT1}{rsfs}{n}{it}

\begin{document}
\title[Gromov-Witten theory of $\text{K3} \times \p^1$]{Gromov-Witten theory of $\text{K3} \times \p^1$ \\ and quasi-Jacobi forms}
\author{Georg Oberdieck}
\address {MIT, Department of Mathematics}
\email{georgo@mit.edu}
\date{\today}
\maketitle

\begin{abstract}
Let $S$ be a K3 surface with primitive curve class $\beta$.
We solve the relative Gromov-Witten theory of $S \times \p^1$ in classes $(\beta,1)$
and $(\beta,2)$. The generating series are
quasi-Jacobi forms and equal to a corresponding
series of genus $0$ Gromov-Witten invariants on the Hilbert scheme of points of $S$.
This proves a special case of a conjecture of Pandharipande and the author.
The new geometric input of the paper is a genus bound for hyperelliptic curves on K3 surfaces proven by Ciliberto and Knutsen.
By exploiting various formal properties we find that a key generating
series is determined by the very first few coefficients.

Let $E$ be an elliptic curve.
As collorary of our computations we prove that Gromov-Witten invariants of $S \times E$
in classes $(\beta,1)$ and $(\beta,2)$
are coefficients of the reciprocal of the Igusa cusp form.
We also calculate several linear Hodge integrals on the moduli space of stable maps to a K3 surface
and the Gromov-Witten invariants of an abelian threefold in classes of type $(1,1,d)$.
\end{abstract}

\setcounter{tocdepth}{1} 
\tableofcontents
\setcounter{section}{-1}

\section{Introduction}
\subsection{Overview}
Let $S$ be a nonsingular projective $K3$ surface,
let $\p^1$ be the projective line, and let $0,1, \infty \in \p^1$ be distinct points.
Consider the relative geometry
\begin{equation} \label{dfgsdg} ( S \times \p^1 ) \ / \ \{ S_0, S_1, S_\infty \} \end{equation}
where $S_z$ denotes the fiber
over the point $z \in \p^1$.

For every $\beta \in H_2(S,\BZ)$ and integer $d \geq 0$,
the pair $(\beta,d)$ determines a class in $H_2(S \times \p^1,\BZ)$ by
\[ (\beta,d) = \iota_{S \ast}(\beta) + \iota_{\p^1 \ast}(d [\p^1]) \]
where $\iota_{S}$ and $\iota_{\p^1}$ are inclusions of
fibers of the projection to $\p^1$ and $S$ respectively.

Let
$\beta_h \in \Pic(S) \subset H_2(S,\BZ)$
be a \emph{primitive} non-zero curve class
satisfying
\[ \langle \beta_h, \beta_h \rangle = 2h-2 \]
with respect to the intersection pairing on $S$.
In \cite{HilbK3, K3xE} the following predictions for the relative Gromov-Witten theory of \eqref{dfgsdg} in classes
$(\beta_h,d)$ were made:
\begin{enumerate}
 \item[(i)] The theory is related by an exact correspondence to the three-point genus~$0$ Gromov-Witten theory of the Hilbert schemes of points of $S$.
 \item[(ii)] For all fixed relative conditions, the generating series of Gromov-Witten invariants
 (summed over the genus and the classes $\beta_h$)
 is a quasi-Jacobi form\footnote{
Jacobi forms are two-parameter generalizations of classical modular forms.
A quasi-Jacobi forms is the holomorphic part of a almost-holomorphic Jacobi form, see \cite{Lib} for the definition
and \cite[Sec.1]{RES} for an introduction.
In this paper we will use the explicit presentation of the quasi-Jacobi form algebra presented in \cite[Appendix B]{HilbK3}.
}.
 \item[(iii)] The theory is governed by an explicit Fock space formalism.
\end{enumerate}
The Jacobi form property of the generating series (part (ii)) is especially striking
since it implies various strong identities and constraints on the curve counting invariants.
In case of the Hilbert scheme of points
an explanation for these symmetries has been found in the invariance of
Gromov-Witten invariants under the monodromies of $\Hilb^d(S)$ in the moduli space of hyperk\"ahler manifolds.
For $S \times \p^1$ the geometric origin of the Jacobi form property is less clear.
Nevertheless, a first hint can be found in the following fact proven by Ciliberto and Knutsen:

\begin{thm}[\cite{CK}, Thm 3.1, Rmk 3.2] \label{thm_CK} Let $\beta$ be a primitive curve class on a K3 surface $S$ such that every
curve in $S$ of class $\beta$ is irreducible and reduced. Then the arithmetic
genus $g = p_a(C)$ of every irreducible curve $C \subset S \times \p^1$ in class $(\beta,d)$ with $d>1$
satisfies
\begin{equation} h \geq g + \alpha \big(g - (d-1)(\alpha+1) \big) \label{CK_equation} \end{equation}
where $\langle \beta, \beta \rangle = 2h-2$ and $\alpha = \floor{\frac{g}{2d-2}}$.
\end{thm}
An elementary check shows \eqref{CK_equation} implies (in fact is equivalent if $d=2$) to the bound
\[ (g+d-1)^2 \leq 4 h (d-1) + (d-1)^2 \,. \]
On the other side the coefficient $c(h,r)$ in the Fourier expansion $\sum_{h, r} c(h,r) q^h y^r$
of a weak Jacobi form of index $d-1$
is non-zero only if
\[ r^2 \leq 4 h (d-1) + (d-1)^2 \,. \]
We find the genus bound by Ciliberto-Knutsen to match
the coefficient bound for weak Jacobi forms under the index shift\footnote{The shift $r = 1-g-d$
is related to a similar shift in the GW/Pairs correspondence \cite{PaPix1, PaPix2}.} $r=1-g-d$.
The appearence of Jacobi forms in the Gromov-Witten theory of $S \times \p^1$ is partly
reflected in the fact that $d$-gonal curves on generic K3 surfaces have many singularities.

One may ask if constraint \eqref{CK_equation} can be used to determine Gromov-Witten invariants of $S \times \p^1$.
The main technical result of the paper shows this is possible in case $d=2$:
For a key choice of incidence condition, the Gromov-Witten invariants
of $S \times \p^1$ in class $(\beta_h, 2)$ are completely determined by formal properties,
the constraint \eqref{CK_equation} and a few calculations in low genus.
By standard techniques this leads to a full evaluation of the relative Gromov-Witten theory
of $S \times \p^1$ in classes $(\beta_h,1)$ and $(\beta_h,2)$
in terms of quasi-Jacobi forms.

\subsection{Relative Gromov-Witten theory of $\text{K3} \times \p^1$} \label{Section:Relative_Gromov_Witten_theory_of_P1K3}
\subsubsection{Definition}
Let $z_1, \dots, z_k$ be distinct points on $\p^1$, and consider the relative geometry
\begin{equation} ( S \times \p^1 )\ / \ \{ S_{z_1} , \ldots , S_{z_k} \}\,. \label{123} \end{equation}
Let $(\beta, d) \in H_2(S \times \p^1, \BZ)$ be a curve class, and let
$\vec{\mu}^{(1)}, \dots, \vec{\mu}^{(k)}$
be ordered partitions of size $d$ with positive parts. The moduli space
\[ \mathbf{M}^{\bullet}_{g,n, (\beta,d), \mathbf{\mu}} = \Mbar^{\bullet}_{g,n}\big( (S \times \p^1) / \{ S_{z_1}, \dots, S_{z_k} \}, (\beta,d), (\vec{\mu}^{(1)},\ldots, \vec{\mu}^{(k)}) \big) \]
parametrizes possibly disconnected\footnote{The moduli space in the disconnected case is always denoted by a $\bullet$ here.}
$n$-pointed relative stable maps of genus $g$ and class $(\beta,d)$
with ordered ramification profile $\vec{\mu}^{(i)}$ along the divisors $S_{z_i}$ respectively.
The relative evaluation maps
\[ \ev_j^{(i)} \colon
\mathbf{M}_{g,n, (\beta,d), \mathbf{\mu}}^{\bullet}
\, \to\,  S_{z_i} \equiv S \,, \quad j=1,\dots, l(\mu_i), \quad i=1,\dots, k \]
send a relative stable map to the $j$-th intersection point with the divisor $S_{z_i}$.
We let $\ev_1, \ldots, \ev_n$ denote the evaluation maps
of the non-relative marked points.

Relative Gromov-Witten invariants are defined using \emph{unordered} relative conditions.
Let $\gamma_1, \dots, \gamma_{24}$ be a fixed basis of $H^{\ast}(S, \BQ)$.
A cohomology weighted partition $\nu$ is a multiset\footnote{The same as a set but with possible repetitions.}
of pairs
\[ \Big\{ (\nu_1, \gamma_{s_1}) , \ldots, (\nu_{l(\nu)}, \gamma_{s_{l(\nu)}}) \Big\} \]
where $\sum_i \nu_i$ is an unordered partition of size $|\nu|$.
The automorphism group $\Aut(\nu)$ consists of the permutation symmetries of $\nu$.

Consider unordered cohomology weighted partitions
\[ \mu^{(1)}, \dots, \mu^{(k)} \,. \]
For every $i \in \{ 1, \ldots, k \}$ let
$( \mu^{(i)}_{j} \,,\, \gamma^{(i)}_{s_j} )_{j=1, \dots, l(\mu_i)}$
be a choice of ordering of $\mu^{(i)}$, and let
$\vec{\mu}^{(i)} = (\mu^{(i)}_{j})$ be the underlying ordered partition.
We define the reduced Gromov-Witten invariants of $(S \times \p^1) / \{ S_{z_i} \}$
with relative conditions $\mu^{(1)}, \dots, \mu^{(k)}$ by
integration over the reduced virtual class\footnote{The
pullback of the symplectic form from the K3 surface yields a trivial quotient of the
standard perfect-obstruction theory on the moduli space. The reduction by this quotient defines the reduced virtual class,
see \cite{KT} for a modern treatment of this process.}
of the moduli space $\mathbf{M}^{\bullet}_{g,n, (\beta,d), \mathbf{\mu}}$:
 \begin{multline*}
\Big\langle \, \mu^{(1)}, \dots, \mu^{(k)} \, \Big| \, \prod_{i=1}^{n} \tau_{\ell_i}(\alpha_i) \Big\rangle^{S \times \p^1/ \{ z_1, \dots, z_k \}, \bullet}_{g, (\beta,d)} \\
=
\frac{1}{\prod_i | \Aut(\mu^{(i)}) |} \cdot
\int_{ [
\mathbf{M}_{g,n, (\beta,d), \mathbf{\mu}}^{\bullet}
]^{\text{red}} }
\prod_{i=1}^{n} \psi_i^{\ell_i} \ev_i^{\ast}(\alpha_i) \cup \prod_{i=1}^{k} \prod_{j=1}^{l(\mu^{(i)})} \ev^{(i) \ast}_j( \gamma^{(i)}_{s_j} ) \,,
\end{multline*}
where $\alpha_1, \dots, \alpha_n \in H^{\ast}(S \times \p^1, \BQ)$
are cohomology classes and $\psi_i$ is the cotangent line class at the $i$th non-relative marked point.
Since all cohomology of $S$ is even, the integral is independent of the choice of ordering of $\mu_i$.
The automorphism factors correct for the choice of an ordering.

\subsubsection{Evaluations}
Let $S$ be a non-singular projective K3 surface
with elliptic\footnote{We work
here with an elliptically fibered K3 surface to obtain a uniform presentation of our results.
By deformation invariance, our results imply parallel statements for
any non-singular projective K3 surface with primitive curve class $\beta$.} fibration $\pi$ and section $s$,
\[ \pi : S \to \p^1, \quad s : \p^1 \to S \,, \quad \pi \circ s = \id_{\p^1} . \]
The class of a fiber of $\pi$ and the image of $s$ are denoted
\[ F, B \, \in\, \Pic(S) \subset H_2(S,\BZ) \]
respectively. Consider the primitive curve classes
\[ \beta_h = B + hF, \quad h \geq 0 \]
of self-intersection $\langle \beta_h, \beta_h \rangle = 2h-2$.
Let also $\pt \in H^4(S, \BZ)$ be the class of a point, and let $\e \in H^0(S,\BZ)$ be the unit.

Consider Gromov-Witten invariants of $S \times \p^1 / \{ 0,1, \infty \}$ with relative conditions
\begin{equation} \label{Rel_Conditions}
\begin{aligned}
\mu_{m,n} & = \{ (1,\pt)^m (1, F)^n \} \\
\nu_{m,n} & = \{ (1, \e)^m (1,F)^{n} \} \\
D(F) & = \{ (1,F) (1, \e)^{m+n-1} \} \,,
\end{aligned}
\end{equation}
over the points $0,1, \infty$ respectively:
\begin{equation} \mathsf{N}_{g,h}(m,n) = \blangle \mu_{m,n} , \nu_{m,n}, D(F) \brangle^{S \times \p^1/\{ 0,1,\infty \}, \bullet}_{g, (\beta,m+n)}. \label{13531523} \end{equation}
By deformation invariance the left hand side only depends on $g,h,m,n$ alone.
The relative condition $D(F)$ over $\infty$ is included to fix the automorphism of $\p^1$ on the target,
but otherwise plays no important role.
The first result of the paper is the complete evaluation of $\mathsf{N}_{g,h}(m,n)$ for which we require several definitions:

Let $u$ and $q$ be formal variables. For $k \geq 1$ let
\begin{equation} C_{2k}(q) = - \frac{B_{2k}}{2k (2k)!} + \frac{2}{(2k)!} \sum_{n \geq 1} \sum_{d | n} d^{2k-1} \, q^n \,, \label{Eisenstein_Series} \end{equation}
denote the classical Eisenstein series, where $B_{2k}$ are the Bernoulli numbers.
We define the Jacobi theta function
\begin{equation}\label{Thetafunction} \Theta(u,q) = u \exp\left( \sum_{k \geq 1} (-1)^{k-1} C_{2k}(q) u^{2k} \right) \end{equation}
and the Weierstrass elliptic function
\[ \wp(u,q) = - \frac{1}{u^2} - \sum_{k \geq 2} (-1)^k (2k-1) 2k C_{2k}(q) u^{2k-2} \,. \]
We will also require the slightly unusual but important function
\[ \mathbf{G}(u,q) = - \Theta(u,q)^2 \big( \wp(u,q) + 2 C_2(q) \big) \,. \]
Finally, define the modular discriminant
\[ \Delta(q) = q \prod_{m \geq 1} (1-q^m)^{24} \,. \]

\begin{thm} \label{mainthm_1b} For all $m \geq 0$ and $n > 0$,
\[ \sum_{g \in \BZ} \sum_{h = 0}^{\infty} \mathsf{N}_{g,h}(m,n) u^{2(g + m + n - 1)} q^{h-1}
=
\frac{1}{m! (n!)^2} \frac{ \mathbf{G}(u,q)^{m} \Theta(u,q)^{2n} }{ \Theta(u,q)^2 \Delta(q)}. \]
If $n=0$ all the invariants $\mathsf{N}_{g,h}(m,n)$ vanish.
\end{thm}
The left hand side of Theorem \ref{mainthm_1b} is a generating series of
relative Gromov-Witten invariants of $S \times \p^1$ in degree $d=m+n$ over $\p^1$.
The right hand side is a (holomorphic) quasi-Jacobi form of index $d-1$.
Theorem~\ref{mainthm_1b} provides an example for the conjectured quasi-Jacobi
form property of generating series of Gromov-Witten invariants of $S \times \p^1$ in all degrees $d$.

The proof of Theorem~\ref{mainthm_1b} proceeds in two steps:
first, the series are computed in degree $d=2$ over $\p^1$ using induction over the genus
via the result of Ciliberto and Knutsen, see Section~\ref{Section:Genus_induction} for details.
Then, the higher degree case follows by a degeneration and localization argument.
The reduction of higher degree to degree $2$ invariants in the second step only works for very limited choices of relative insertions
and is one of the reasons to restrict to the considered case.
Closed evaluations in higher degree with more general insertions require new methods.

\subsection{Hilb/GW correspondence}
Let
\[ \Hilb^d(S) \]
be the Hilbert scheme of $d$ points of the K3 surface $S$.

For all $\alpha \in H^{\ast}(S ; \BQ)$ and $i > 0$ let 
\[ \Fp_{-i}(\alpha) : H^{\ast}(\Hilb^d(S);\BQ) \to H^{\ast}(\Hilb^{d+i}(S);\BQ), \ \gamma \mapsto \Fp_{-i}(\alpha) \gamma \]
be the Nakajima creation operator
obtained by adding length $i$ punctual subschemes incident to a cycle Poincare dual to $\alpha$.
The cohomology of $\Hilb^d(S)$ is completely described by the cohomology of $S$
via the action of the operators $\Fp_{-i}(\alpha)$ on the vacuum vector 
\[ \vacuum \in H^{\ast}(\Hilb^0(S);\BQ) \equiv \BQ. \]

To every cohomology weighted partition $\mu = \{ (\mu_i, \gamma_{s_i}) \}$ of size $d$ we associate the class
\[ | \mu \rangle = \frac{1}{\Fz(\mu)} \prod_{i} \Fp_{-i}(\gamma_{s_i}) \vacuum \]
in $H^{\ast}(\Hilb^d(S), \BQ)$, where $\Fz(\mu) = |\Aut(\mu)| \prod_{i}\mu_i$.

Let $\beta \in H_2(S,\BZ)$ be a non-zero curve class on $S$.
The associated curve class on $\Hilb^d(S)$,
defined as the Poincare dual to
\[ \Fp_{-1}(\beta) \Fp_{-1}(\pt)^{d-1} \vacuum \,, \]
is denoted by $\beta$ as well. We will also require $A \in H_2(\Hilb^d(S),\BZ)$, the class of an exceptional curve
Poincare dual to
\[ \Fp_{-2}(\pt) \Fp_{-1}(\pt)^{d-2} \vacuum \,. \]

Let $\lambda_1, \dots, \lambda_r$ be cohomology weighted partitions and let
\begin{equation} \label{HilbK3invariants}
\big\langle \lambda_1, \dots, \lambda_n \big\rangle^{\Hilb^d(S)}_{0,\beta + kA}
=
\int_{[ \Mbar_{0,n}(\Hilb^d(S),\beta + kA) ]^{\text{red}}} \prod_{i=1}^{n} \ev_i^{\ast}(\lambda_i)
\end{equation}
be the reduced genus $0$ Gromov-Witten invariants of $\Hilb^d(S)$ in class $\beta + k A$ \cite{HilbK3}.
The following GW/Hilb correspondence was conjectured in \cite{K3xE}.

\begin{conj}[\cite{K3xE}] \label{asdasd} \label{GW/Hilb_correspondence}
For primitive $\beta$,
\begin{multline} \label{eqn_correspondence}
(-1)^d \sum_{k \in \BZ} \big\langle \mu, \nu, \rho \big\rangle^{\Hilb^d(S)}_{0,\beta + kA} y^k \\
\ = \  
(-iu)^{l(\mu) + l(\nu) + l(\rho) - d}
\sum_{g \geq 0} \blangle \mu , \nu , \rho \brangle^{S \times \p^1 / \{ 0,1,\infty \}}_{g, (\beta,d)} u^{2g-2}
\end{multline}
under the variable change $y = - e^{iu}$.
\end{conj}

The Gromov-Witten invariants of $\Hilb^d(S)$ which correspond
to the invariants $\mathsf{N}_{g,h}(m,n)$
were calculated in \cite{HilbK3}.
The result exactly matches the evaluation of Theorem~\ref{mainthm_1b} under the correspondence of Conjecture \ref{GW/Hilb_correspondence}.
Hence Theorem \ref{mainthm_1b} gives an example of Conjecture \ref{GW/Hilb_correspondence} in every degree~$d$.
For low degree we have the following result:

\begin{thm} \label{GWHilb_thm} \label{thm_GWHilb_correspondence} Conjecture \ref{asdasd} holds if $d=1$ or $d=2$. \end{thm}

Let $(z,\tau) \in \BC \times \BH$.
The ring $\QJac$ of quasi-Jacobi forms is the linear subspace
\[ \QJac \subset \BQ[ \Theta(z,\tau), C_2(\tau), C_4(\tau), \wp(z,\tau), \wp^{\bullet}(z,\tau), J_1(z,\tau)] \]
of functions which are holomorphic at $z=0$;
here $\Theta$ is the Jacobi theta function, 
$\wp$ is the Weierstrass elliptic function,
$\wp^{\bullet}$ is its derivative with respect to $z$,
and $J_1$ is the logarithmic derivative of $\Theta$ with respect to $z$, see \cite[Appendix B]{HilbK3}.
The space $\QJac$ is naturally graded by index $m$ and weight $k$,
\[ \QJac = \bigoplus_{m \geq 0} \bigoplus_{k \geq -2m} \QJac_{k,m}, \]
with finite-dimensional summands $\QJac_{k,m}$.
We identify a quasi Jacobi form $\psi(z, \tau)$ with its power series expansions in the variables
$q = e^{2 \pi i \tau}$ and $u = 2 \pi z$.

In \cite{HilbK3} the invariants of $\Hilb^2(S)$ have been completely determined in the primitive case.
In particular, combining \cite[Theorem 3]{HilbK3} and Theorem~\ref{GWHilb_thm} we have the following.

\begin{cor} Let $\mu, \nu, \rho$ be cohomology weighted partitions of size $2$.
Then, under the variable change $u = 2 \pi z$ and $q = e^{2 \pi i \tau}$, we have
\[
(-iu)^{l(\mu) + l(\nu) + l(\rho) - d}
\sum_{g, h} \blangle \mu, \nu, \rho \brangle^{S \times \p^1 / \{ 0,1,\infty \}, \bullet}_{g, (\beta_h,2)} u^{2g-2} q^{h-1}
\ = \ 
\frac{\psi(z,\tau)}{\Delta(q)}
\]
for a quasi-Jacobi form $\psi(z,\tau)$ of index $1$ and weight $\leq 8$.
\end{cor}

\subsection{The product $\text{K3} \times E$} \label{Section_Introduction_K3xE}
Let $S$ be a nonsingular projective $K3$ surface, and let
$E$ be a nonsingular elliptic curve.
The 3-fold
$$X=S \times E$$
has trivial canonical bundle, and hence is Calabi-Yau.
%
Let $\beta \in H_2(S,\mathbb{Z})$ be an effective curve class
and let $d\geq 0$ be an integer.
The pair $(\beta,d)$ determines a class in $H_2(X,\mathbb{Z})$ by
$$(\beta,d)= \iota_{S*}(\beta)+ \iota_{E*}(d[E])$$
where $\iota_S$ and $\iota_E$ are inclusions of fibers of the projections to $E$ and $S$ respectively.

The moduli space $\Mbar^{\bullet}_{g}(X, (\beta,d))$ of disconnected genus $g$ stable maps in class $(\beta,d)$ carries a reduced virtual class
\[ [ \Mbar^{\bullet}_{g}(X, (\beta,d)) ]^{\text{red}} \]
of dimension $1$. The group $E$ acts on the moduli space by translation
and the dimension of the reduced class correspond to the $1$-dimensional orbits under this action.
We define a count of curves in $X$ by
imposing an incidence condition which selects one point in each $E$ orbit.
Concretely, let $\omega \in H^2(E, \BZ)$ be the class of a point and let
$\beta^{\vee} \in H^2(S, \BQ)$ be a class satisfying $\langle \beta, \beta^{\vee} \rangle = 1$.
We define
\begin{equation} \label{15343}
\mathsf{N}_{g,(\beta,d)}^X = \int_{[ \Mbar^{\bullet}_{g,1}(X, (\beta,d)) ]^{\text{red}}} \ev_1^{\ast}( \pi_1^{\ast}(\beta^{\vee}) \cup \pi_2^{\ast}(\omega) ) \,.
\end{equation}
A complete evaluation of $\mathsf{N}_{g,(\beta,d)}^X$
was conjectured in \cite{K3xE}
matching the physical predictions \cite{KKV}.
We consider here the case of primitive $\beta$.

For primitive $\beta_h$ the integral
\eqref{15343} only depends on the norm $\langle \beta_h, \beta_h \rangle = 2h-2$.
We write
\[ \mathsf{N}_{g,h,d}^X = \mathsf{N}_{g,(\beta_h,d)}^X \,. \]
\begin{conj}[\cite{K3xE}] \label{fjgdf} Let $\tilde{q}$ be a formal variable. Then
\begin{equation*} \label{chi10}
\sum_{d = 0}^{\infty} \sum_{h = 0}^{\infty} \sum_{g \in \BZ} \mathsf{N}_{g,h,d}^X u^{2g-2} q^{h-1} \tilde{q}^{d-1} = \frac{1}{\chi_{10}(u,q,\tilde{q})}
\end{equation*}
where $\chi_{10}(u,q, \tilde{q})$ is the Igusa cusp form in the notation of \cite{K3xE}.
\end{conj}

Conjecture \ref{fjgdf} contains several known cases. In curve class $(\beta, 0)$ the
invariant $\mathsf{N}^X_{g, h,d}$ reduces to the Katz-Klemm-Vafa formula proven in \cite{MPT}.
The case $(\beta_{0}, d)$ for $d \geq 0$ reduces to the product
of a $\mathcal{A}_1$-resolution times an elliptic curve computed in \cite{M}.
The cases $(\beta_0, d)$ and $(\beta_1, d)$ have been recently obtained
by J.~Bryan \cite{Bryan-K3xE}.
Here we show the cases $(\beta_h, 1)$ and $(\beta_h,2)$ of Conjecture \ref{fjgdf}:

\begin{thm} \label{Theorem_K3xE}
For $d=1$ and $d=2$ we have
\[ \sum_{h \geq 0} \sum_{g \in \BZ} \mathsf{N}_{g,h,d}^X u^{2g-2} q^{h-1} = \left[ \frac{1}{\chi_{10}(u,q,\tilde{q})} \right]_{\tilde{q}^{d-1}} \]
where $\left[ \ \cdot \ \right]_{\tilde{q}^k}$ denotes the coefficient of $\tilde{q}^k$.
\end{thm}

\subsection{Abelian threefolds}
Consider a complex abelian variety $A$ of dimension $3$,
and let $\beta \in H_2(A, \BZ)$ be a curve class of type
\[ (d_1, d_2, d_3), \quad d_1, d_2 > 0,\ d_3 \geq 0 \,, \]
where the type is obtained from the standard divisor theory of the dual abelian variety $A^{\vee}$.
Since $d_1, d_2 > 0$, the action of $A$ on the moduli space $\Mbar_{g}(A,\beta)$
by translation has finite stabilizers and the stack quotient
\[ \Mbar_{g}(A,\beta) / A \]
is Deligne-Mumford.
A $3$-reduced virtual class
$[ \Mbar_{g}(A, \beta) / A ]^{\text{3-red}}$
of dimension $0$
has been defined in \cite{BOPY} and gives rise to Gromov-Witten invariants
\begin{equation} \mathsf{N}^A_{g,(d_1, d_2, d_3)} = \int_{[ \Mbar_{g}(A, \beta)/A ]^{\text{3-red}} } 1 \label{aaa} \end{equation}
counting genus $g$ curves in $A$ of class $\beta$ \emph{up to translation}.

In genus $3$, the counts $\mathsf{N}^A_{3,(d_1, d_2, d_3)}$ reduce
to a lattice count in abelian groups \cite{Debarre, Gottsche, LS}
A full formula for $\mathsf{N}^A_{g,(d_1, d_2, d_3)}$ in case $d_1 = 1$
was recently conjectured in \cite{BOPY} based on new calculations
of the Euler characteristic of the Hilbert scheme of curves in $A$.
The following result verifies this conjecture in case $d_1 = d_2 = 1$.

\begin{thm} \label{abelianthm}
\[ \sum_{d = 0}^{\infty} \sum_{g=2}^{\infty} \mathsf{N}^A_{g,(1,1,d)} u^{2g-2} q^d \, = \, \Theta(u,q)^2 \]
\end{thm}

  An interesting question is to explore the enumerative significance of Theorem \ref{abelianthm}.
  Define BPS numbers $\mathsf{n}_{g, (1, d, d')}$ by the expansion
  \[ \sum_{g} \mathsf{n}_{g, (1,d,d')} (2 \sin(u/2))^{2g-2} = \sum_{g \geq 0} \mathsf{N}^A_{g,(1,d,d')} u^{2g-2} \,. \]
  Then it is natural to ask: If $A$ is a generic abelian threefold carrying a curve class $\beta$ of type $(1,d,d')$,
  do there exist only finitely many isolated curves of genus $g$ in class $\beta$ up to translation?
  Is every such curve non-singular? If both questions can be answered affirmative,
  the BPS numbers $\mathsf{n}_{g, (1,d,d')}$ are enumerative.

\subsection{Plan of the paper}
In Section~\ref{Section_The_bracket_notation} we review a bracket notation for Gromov-Witten invariants.
In Section~\ref{Section:First_vanishing} we exploit a basic evaluation
of Gromov-Witten invariants of $S \times \p^1$ in classes $(\beta_h,1)$,
leading to a proof of Theorem~\ref{abelianthm} on abelian threefolds.
In Section~\ref{Section_Formal_series_of_quasimodular_forms} we prove a uniqueness
statement for formal series of quasi-modular forms.
Section~\ref{Section:Genus_induction} is the heart of the paper:
here we apply the result of Ciliberto and Knutsen on hyperelliptic curves in K3 surface
to calculate a key generating series of Gromov-Witten invariants of $S$.
In Section~\ref{Section:Relative_Invariants_of_P1K3}
we apply standard techniques to solve for the relative Gromov-Witten theory
of $S \times \p^1$ in degrees $d=1$ and $d=2$.
As a result we obtain the GW/Hilb correspondence (Theorem \ref{GWHilb_thm})
and Theorem \ref{Theorem_K3xE} on the Gromov-Witten theory of the $S \times E$.

\subsection{Acknowledgements}
I would like to thank Davesh Maulik for interesting discussions and technical assistence,
and Jim Bryan, Tudor Padurariu, Rahul Pandharipande,
Aaron Pixton, Martin Raum, Junliang Shen, and Qizheng Yin
for discussions about counting curves in K3 geometries.
I am also very grateful to the referees for a careful reading of the manuscript and their comments.
A great intellectual debt is owed to the paper \cite{MPT} by Maulik, Pandharipande and Thomas,
where many of the techniques used here were developed.

\section{The bracket notation} \label{Section_The_bracket_notation}
Let $X$ be a smooth projective variety and let $\beta \in H_2(X,\BZ)$ be a curve class.
We will denote connected Gromov-Witten invariants of $X$ by the
bracket notation
\begin{equation} \label{def_gw}
\Big\langle \, \alpha \, ; \, \tau_{k_1}(\gamma_1) \cdots \tau_{k_n}(\gamma_n) \Big\rangle^X_{g, \beta}
\ = \
\int_{[ \Mbar_{g,n}(X ,\beta) ]^{\text{vir}}} \alpha \cup \prod_{i=1}^{n} \ev_i^{\ast}(\gamma_i) \psi_i^{k_i},
\end{equation}
where
\begin{itemize}
 \item $\Mbar_{g,n}(X,\beta)$ is the moduli space of connected $n$-marked stable maps of genus $g$ and class $\beta$,
 \item $\gamma_1, \ldots, \gamma_n \in H^{\ast}(X)$ are cohomology classes,
 \item $\alpha$ is a cohomology class on $\Mbar_{g,n}(X,\beta)$, usually taken to be the pullback
 of a \emph{tautological} class \cite{FP13} under the forgetful map $\Mbar_{g,n}(X,\beta) \to \Mbar_{g,n}$
 to the moduli space of curves.
\end{itemize}
If the obstruction sheaf on $\Mbar_{g,n}(X, \beta)$ admits a trivial
quotient obtained from a holomorphic $2$-form on $X$,
the integral in \eqref{def_gw} is assumed to be over the \emph{reduced} virtual class.
For abelian threefolds we will use the $3$-reduced virtual class \cite{BOPY}.
The parallel definition of \eqref{def_gw} for disconnected invariants is denoted by attaching the superscript $\bullet$
to the bracket and the moduli spaces.

Let $\BE \to \Mbar_{g,n}(X, \beta)$ (resp. $\BE \to \Mbar_{g,n}^{\bullet}(X, \beta)$) be the Hodge bundles with fiber $H^0(C, \omega_C)$ over the moduli point $[f : C \to X]$.
The total Chern class of the dual of $\BE$,
\[ \BE^{\vee}(1) = c(\BE^{\vee}) = 1 - \lambda_1 + \ldots + (-1)^g \lambda_g, \]
is often used for the insertion $\alpha$. 

We extend the bracket \eqref{def_gw} by multilinearity in the insertions.
Since for dimension reasons only finitely many terms contribute, the formal expansion
\[ \frac{\gamma}{1 - \psi_i} = \sum_{k=0}^{\infty} \tau_k(\gamma), \quad \gamma \in H^{\ast}(X) \,. \]
is well-defined.

Assume that $X$ admits a fibration
\[ \pi : X \to \p^1 \]
and let $X_0, X_\infty$ be the fibers of $\pi$ over the points $0, \infty \in \p^1$.
We will use the standard bracket notation
\[
\Big\langle \ \mu\ \Big| \ \alpha\, \prod_{i} \tau_{k_i}(\gamma_i) \ \Big| \ \nu \ \Big\rangle^{X}_{g, \beta}
 = \int_{ [\overline{M}_{g,n}( X /\{ X_0, X_\infty\}, \beta)_{\mu, \nu}]^{\text{vir}} } \alpha\, \cup\, \prod_{i}  \psi_i^{k_i} \ev_i^{\ast}(\gamma_i)
\]
for the Gromov-Witten invariants of $X$ relative to the fibers $X_0$ and $X_{\infty}$.
The integral is over the moduli space of stable maps
 \[ \overline{M}_{g,n}( X /\{ X_0, X_{\infty} \},\, \beta) \]
relative to the fibers over $0,\infty \in \p^1$ in class $\beta$.
Here, $\mu$ and $\nu$ are unordered cohomology weighted partitions,
weighted by cohomology classes on $X_0$ and $X_{\infty}$ respectively\footnote{We follow the convention of Section~\ref{Section:Relative_Gromov_Witten_theory_of_P1K3}
or equivalently of \cite{MP}.}.
The integrand contains the cohomology class $\alpha$ and the descendents.
Again, we use a \emph{reduced} virtual class whenever possible.

We will form generating series of the absolute and relative invariants
above. Throughout we will use the following conventions:

In K3 geometries we assign to a primitive class $\beta_h$ of norm $\langle \beta_h , \beta_h \rangle = 2h-2$
the variable $q^{h-1}$.
The $d$-times multiple of the fundamental class of an elliptic curve (in a trivial elliptic fibration) will correspond to $q^d$.
For absolute invariants the genus $g$ Gromov-Witten invariant in class $\beta$
will be weighted by the variable
\[ u^{2g - 2 + \int_\beta c_1(X)}. \]
For relative invariants with relative conditions specified by cohomology weighted partitions $\mu_1, \dots, \mu_k$
we will use
\[ u^{2g-2+\int_\beta c_1(X) + \sum_{i=1}^{k} l(\mu_i) - |\mu_i|} \,. \]
For example, in case of the elliptically fibered K3 surface $S$ with curve classes $\beta_h = B + hF$ we will use
\begin{multline} \label{Generating_Series_K3_Surfaces}
\Big\langle \alpha ; \tau_{k_1}(\gamma_1) \cdots \tau_{k_n}(\gamma_n) \Big\rangle^{S}
= \sum_{g \geq 0} \sum_{h \geq 0} \big\langle \alpha ; \tau_{k_1}(\gamma_1) \cdots \tau_{k_n}(\gamma_n) \big\rangle^S_{g, \beta_h} u^{2g-2} q^{h-1} \,.
\end{multline}

\section{Calculations in degree $1$} \label{Section:First_vanishing}
\subsection{Overview}
We evaluate a special Gromov-Witten invariant on $\text{K3} \times \p^1$
in class $(\beta_h,1)$.
By the Katz-Klemm-Vafa formula this leads to a proof of Theorem \ref{abelianthm}.

\subsection{Evaluation} \label{Subsection_Proof_of_Theorem_1}
Let $S$ be a K3 surface, let $\beta_h \in H_2(S, \BZ)$ be a primitive curve
class satisfying $\langle \beta_h, \beta_h \rangle =2h-2$ and let
\[ F \in H^2(S,\BZ) \]
be a class satisfying $F \cdot \beta_h = 1$ and $F \cdot F = 0$.

Let $\omega \in H^2(\p^1)$ be the class of a point, and let
\[ F \boxtimes \omega = \pi_1^{\ast}(F) \cup \pi_2^{\ast}(\omega) \in H^4(S \times \p^1) \]
where $\pi_i$ is the projection from $S \times \p^1$ to the $i$th factor.
Consider the connected Gromov-Witten invariant
\begin{equation} \blangle \tau_0(F \boxtimes \omega)^3 \brangle_{g,(\beta_h, 1)}^{S \times \p^1}
 = \int_{[ \Mbar_{g,3}(S \times \p^1, (\beta_h, 1)) ]^{\text{red}}}
 \prod_{i=1}^{3} \ev_i^{\ast}(F \boxtimes \omega) \,.
\label{rwerw} \end{equation}

\begin{prop} \label{vanishing1} For every $h \geq 0$, we have
\[
\blangle \tau_0(F \boxtimes \omega)^3 \brangle_{g,(\beta_h, 1)}^{S \times \p^1}
=
\begin{cases}
\left[ \frac{1}{\Delta(q)} \right]_{q^{h-1}} & \text{ if } g = 0 \\
0 & \text{ if } g > 0 \,,
\end{cases}
\]
where $[\ \cdot\ ]_{q^{n}}$ denotes extracting the $n$-th coefficient.
\end{prop}
\begin{proof}
We may take $S$ to be generic and $\beta_h$ to be irreducible.
Let $F_i$, $i=1,2,3$ be generic distinct smooth submanifolds of class $F$ which intersect all rational curves in
class $\beta_h$ transversely in a single point. Let also $x_1, x_2, x_3$ be distinct points in $\p^1$.
The products
\[ F_i \times x_i \, \subset \, S \times \p^1, \quad i=1,2,3 \]
have class $F \boxtimes \omega$.

Consider an algebraic curve $C \subset S \times \p^1$ in class $(\beta_h, 1)$ incident to $F_i \times x_i$ for all $i$.
Since $F_i \cap F_j = \varnothing$ for $i \neq j$,
the curve $C$ is irreducible and reduced.
Because the projection $C \to \p^1$ is of degree~$1$ the curve
$C$ is non-singular. Since irreducible rational curves on K3 surfaces
are rigid,
the only deformations of $C$ in $S \times \p^1$ are
by translations by automorphisms of $\p^1$. The incidence
conditions $F_i \times x_i$ then select precisely one member of each translation class.

We find curves in class $(\beta_h,1)$ incident to all $F_i \times x_i$ are in $1$-to-$1$ correspondence
with rational curves on $S$ in class $\beta_h$.
By the Yau-Zaslow formula proven in \cite{BL, Bea99, Chen} there are precisely
\[ \left[ \frac{1}{\Delta(q)} \right]_{q^{h-1}} \]
such curves. It remains to calculate their contribution to \eqref{rwerw}.

By arguments parallel to the proof of \cite[Proposition 5]{K3xE}
the generating series of \eqref{rwerw} over all $g$ is related to
the generating series of reduced stable pair invariants of $S \times \p^1$
in class $(\beta_h, 1)$ with incidence conditions $F_i \times x_i, i=1,2,3$.
The contribution of the isolated curve $C \subset S \times \p^1$ to the stable pair invariant is
obtained from a direct modification of the calcation in \cite[Section 4.2]{PT1} to the reduced setting \cite{MPT}.
Translating back to Gromov-Witten theory we find each curve $C$ contributes $1$ in genus $0$ and $0$ otherwise.
This concludes the proof.
\end{proof}

\subsection{Relative theory of $\p^1 \times E$} \label{Section_Relative_Theory_of_P1E}
Let $E$ be an elliptic curve and consider the curve class
\[ (1,d) = \iota_{\p^1 \ast}( [ \p^1 ] ) + \iota_{E \ast}( d[E] ) \in H_2(\p^1 \times E, \BZ) \]
where $\iota_{\p^1}, \iota_{E}$ are the inclusion of fibers
of the projections to the second and first factor respectively.
We will use the generating series of relative invariants of $\p^1 \times E$,
\begin{multline} \label{gen_series_p1e}
\Big\langle \ \mu \ \Big| \ \alpha\, \prod_{i} \tau_{a_i}(\gamma_i) \ \Big| \ \nu \ \Big\rangle^{\p^1 \times E}
=\ \sum_{g \geq 0} \sum_{d \geq 0}
\Big\langle \ \mu \ \Big| \ \alpha\, \prod_{i} \tau_{a_i}(\gamma_i) \ \Big| \ \nu \ \Big\rangle^{\p^1 \times E}_{g, (1,d)}
\, u^{2g}q^d \,.
\end{multline}
Since the class $(1,d)$ is of degree $1$ over $\p^1$, the
relative insertions $\mu$ and $\nu$ are cohomology classes on the fibers:
\[ \mu \in H^{\ast}(0\times E) \quad \text{ and } \quad \nu \in H^{\ast} (\infty \times E) \,. \]
Similar definitions apply also to the case of a single relative divisor.

\begin{lemma} \label{P1xE_Lemma} \hfill
\begin{enumerate}
 \item[(a)] The series \eqref{gen_series_p1e} vanishes unless 
 \[ \deg_{\BR}(\mu) + \deg_{\BR}(\nu) \leq 2 \,, \]
 where $\deg_{\BR}(\gamma)$ denotes the real degree of $\gamma$. 
 \item[(b)] We have $\blangle\, \omega\, |\, \hodge \brangle^{\p^1 \times E}
 = \blangle\, \e\, |\, \hodge \tau_0(\pt) \brangle^{\p^1 \times E}
 = 1$.
 \item[(c)] Let $D = q \frac{d}{dq}$. Then,
 \[ \blangle\, \omega\, |\, \hodge \tau_0(\pt) \, \brangle^{\p^1 \times E} = \frac{D \Theta(u,q)}{\Theta(u,q)} \,. \]
\end{enumerate}
\end{lemma}
\begin{proof} (a) follows since a curve $C \subset \p^1 \times E$ in class $(1,d)$ is of the form
\[ (\p^1 \times e)\, +\, D \]
where $e \in E$ is a fixed point and $D$ is a fiber
of the projection $\p^1 \times E \to \p^1$.
Hence for every relative stable map $f$ to $\p^1 \times E / \{ 0, \infty \}$
the intersection point over $0$ and over $\infty$ agree,
which implies the claim (for example choose cycles representing $\mu$ and $\nu$).
Part (b) is \cite[Lemma 24]{MPT} and
part (c) follows from \cite[Lemma 26]{MPT}.
\end{proof}

\subsection{Fiber integrals}
Let $S$ be the elliptically fibered K3 surface with curve class $\beta_h = B+  hF$
where $B,F$ are the section and fiber class respectively.
Recall also the notation \eqref{Generating_Series_K3_Surfaces}.
\begin{prop} \label{Proposition_Thm1m=0}
$\displaystyle \Big\langle\, \BE^{\vee}(1) \ \prod_{i=1}^{n} \frac{F}{1-\psi_i} \Big\rangle^{S}
=
\frac{1}{u^{2n}} \frac{\Theta(u,q)^{2n}}{\Theta(u,q)^2 \Delta(q)}
$
\end{prop}
\begin{proof}
By Proposition \ref{vanishing1} we have
\begin{equation*}
\sum_{g \geq 0} \sum_{h \geq 0} \blangle \tau_0(F \boxtimes \omega)^3 \brangle_{g,(\beta_h,1)}^{S \times \p^1} u^{2g} q^{h-1} = \frac{1}{\Delta(q)} \,.
\label{gen1}
\end{equation*}
The factor $\p^1$ admits an action of $\BC^{\ast}$ which lifts to the moduli space $\Mbar_{g,n}(S \times \p^1, (\beta_h,1))$.
Applying the virtual localization formula \cite{GP} and using the divisor axiom yields
\begin{equation*} \label{gen2}
\blangle \tau_0(F \boxtimes \omega)^3 \brangle_{g,(\beta_h,1)}^{S \times \p^1}
= \Big\langle\, \BE^{\vee}(1) \ \frac{F}{1 - \psi_1} \Big\rangle^S_{g, \beta_h} \,.
\end{equation*}
This proves the claim for $n=1$.

For the general case, we degenerate $S$ to the normal cone of a fiber $E$ of the elliptic fibration $S \to \p^1$,
\begin{equation} \label{degeneration} S\ \leadsto\ S \, \cup  \, ( \p^1 \times E ) \end{equation}
specializing the fiber class $F$ to the $\p^1 \times E$ component.
The degeneration formula \cite{Junli1, Junli2},
see also \cite[Section 6]{MPT} and \cite[Section 3.4]{BOPY} for the modifications in the reduced case,
yields
\begin{equation} \label{4242}
\Big\langle \BE^{\vee}(1) \, \prod_{i=1}^{n} \frac{F}{1-\psi_i} \Big\rangle^S =
\Big\langle\,  \BE^{\vee}(1)\ \Big| \ 1 \ \Big\rangle^S
\Big\langle\ \omega\ \Big|\ \BE^{\vee}(1) \prod_{i=1}^{n} \frac{F}{1 - \psi_i}\, \Big\rangle^{\p^1 \times E}
\end{equation}
We analyze both terms on the right hand side.
By a further degeneration of $S$ (using Lemma \ref{P1xE_Lemma}) and then using the Katz-Klemm-Vafa formula \cite{MPT} we get
\begin{equation} \label{KKV}
\Big\langle\  \BE^{\vee}(1)\ \Big| \ 1 \ \Big\rangle^S = \Big\langle\,  \BE^{\vee}(1) \, \Big\rangle^S = \frac{1}{\Theta(u,q)^2 \Delta(q)} \,.
\end{equation}

For the second term,
we degenerate the base $\p^1$ to obtain a chain of $n+1$ surfaces
isomorphic to $\p^1 \times E$. The first $n$ of these each receive a single insertion $F$
weighted by psi classes.
Using Lemma \ref{P1xE_Lemma} we obtain
\begin{equation} \label{4343}
\Big\langle\ \omega\ \Big|\ \prod_{i=1}^{n} \frac{F}{1 - \psi}\, \Big\rangle^{\p^1 \times E}
=
\left( \Big\langle\ \omega\ \Big|\ \frac{F}{1 - \psi} \Big|\ 1\ \Big\rangle^{\p^1 \times E} \right)^n \,.
\end{equation}

In case $n=1$ the left hand side of \eqref{4242} is known and we can solve for \eqref{4343}. The result is
\begin{equation} \label{assd}
\Big\langle\ \omega\ \Big|\ \BE^{\vee}(1) \, \frac{F}{1 - \psi}\ \Big|\ 1\ \Big\rangle^{\p^1 \times E}
=
\frac{\Theta(u,q)^2}{u^2} \,. \end{equation}
Inserting \eqref{KKV} and \eqref{assd} back into \eqref{4242}, the proof is complete.
\end{proof}

\subsection{The abelian threefold}
Recall the bracket notation \eqref{gen_series_p1e} for the generating series
of relative invariants of $\p^1 \times E$ in class $(1,d)$.
We will need the following result.
\begin{lemma} \label{riogkegfr} For $\pt \in H^4(\p^1 \times E, \BZ)$ the point class,
\[
\Big\langle \ 1\  \Big| \ \BE^{\vee}(1) \ \frac{\pt}{1-\psi} \ \Big\rangle^{\p^1 \times E} = \frac{\Theta(u,q)^2}{u^2}
\]
\end{lemma}
\begin{proof}
The translation action of the elliptic curve on $\p^1 \times E$
yield basic vanishing relations on the Gromov-Witten theory of $\p^1 \times E$,
see \cite[Section 3.3]{BOPY} for the parallel case of abelian surfaces and also \cite{OP3}.
A straightforward application here yields
\[
\Big\langle \ 1\  \Big| \ \BE^{\vee}(1) \ \frac{\pt}{1-\psi} \Big\rangle^{\p^1 \times E}
=
\Big\langle\ \omega\ \Big|\ \BE^{\vee}(1) \ \frac{F}{1 - \psi_1} \, \Big\rangle^{\p^1 \times E} \,,
\]
where $F$ is the fiber over a point in $\p^1$. The claim follows from \eqref{assd}.
\end{proof}
%
%
%
%

\begin{proof}[Proof of Theorem \ref{abelianthm}]
By deformation invariance we may consider the special geometry
\[ A = E_1 \times E_2 \times E_3 \,, \]
where $E_i$ are elliptic curves, and the curve classes
\[ (1,1,d) = \iota_{E_1, \ast}([E_1]) + \iota_{E_2 \ast}([E_2]) + \iota_{E_3 \ast}( d [E_3] )\ \in H_2(A, \BZ) \]
where $\iota_{E_i} : E_i \hookrightarrow A$ is the inclusion of a fiber of the map forgetting the $i$th factor.
For $i \in \{ 1, 2, 3 \}$ let
\[ H_i \in H^2(A) \]
be the pullback of the point class from the $i$-th factor of $A$.
By \cite[Lemma 18]{BOPY},
\[ \mathsf{N}^{A}_{g, (1,1,d)} = \frac{1}{2} \Big\langle \tau_0( \pt ) \tau_0(H_1 H_2) \Big\rangle^{A,\, 3\text{-red}}_{g, (1,1,d)} \]
where the right hand side are absolute 3-reduced Gromov-Witten invariants of $A$
with insertions the point class $\pt \in H^6(A, \BZ)$ and $H_1 H_2$.

We degenerate the factor $E_1$ to a nodal rational curve and resolve.
Applying the degeneration formula modified to the reduced
case\footnote{This is parallel to the breaking
of the reduced virtual class in the K3 case when degenerating
to two rational elliptic surfaces, see \cite[Section 4]{MPT}.}
we obtain
\[
\big\langle \tau_0( \pt ) \tau_0(H_1 H_2) \big\rangle^{A,\text{3-red}}_{g, \beta}
=
\Big\langle\ 1\ \Big| \ \tau_0( \pt ) \tau_0(H_1 H_2)\ \Big|\ 1\ \Big\rangle^{\p^1 \times E_2 \times E_3, \text{red}}_{g-1, (1,1,d)} \,,
\]
where the right hand side are $1$-reduced invariants of $\p^1 \times E_2 \times E_3$
relative to the fibers over $0$ and $\infty$, and $H_i$ is
the pullback of the point class from the $i$-th factor.

By a degeneration of the base $\p^1$ to a chain of three $\p^1$'s and specializing all
insertions to the middle factor, we obtain
\begin{equation*}
\Big\langle\ 1\ \Big| \ \tau_0( \pt ) \tau_0(H_1 H_2)\ \Big|\ 1\ \Big\rangle^{\p^1 \times E_2 \times E_3, \text{red}}_{g-1, (1,1,d)}
\ = \
\Big\langle \tau_0( \pt ) \tau_0(H_1 H_2) \Big\rangle^{\p^1 \times E_2 \times E_3, \text{red}}_{g-1, (1,1,d)} \,,
\label{12345}
\end{equation*}
to which we apply the localization formula to get
\begin{equation} \label{midd}
\Big\langle\ \BE^{\vee}(1)\, \frac{\pt}{1-\psi_1}\, \Big\rangle^{E_2 \times E_3, \text{red}}_{g-1, (1,d)}
+ \Big\langle\ \BE^{\vee}(1)\, \frac{H_2}{1-\psi_1} \, \tau_0(\pt) \Big\rangle_{g-1,(1,d)}^{E_2 \times E_3, \text{red}} \,.
\end{equation}

Let $E = E_3$. We calculate both terms of \eqref{midd} by the degeneration
formula for
\[ E_2 \times E \ \leadsto \ (E_2 \times E) \cup (\p^1 \times E) \,. \]
where the point and $H_2$ class are specialized to the $\p^1 \times E$ component.
In both cases we will use the evaluation
$\langle\, \BE^{\vee}(1)\, | \, \omega \, \rangle^{E_2 \times E_3}_{g, (1,d)} = \delta_{g,1} \delta_{d,0}$ 
proven in \cite[Lemma 8]{BOPY}. The result for the first term is
\[
\Big\langle\ \BE^{\vee}(1)\, \frac{\pt}{1-\psi_1}\, \Big\rangle^{E_2 \times E_3, \text{red}}_{g-1, (1,d)}
=
\Big\langle\ 1 \ \Big|\ \BE^{\vee}(1) \frac{\pt}{1-\psi_1} \Big\rangle^{\p^1 \times E}_{g-2, (1,d)} \,,
\]
and similarly the second term yields
\begin{align*}
\Big\langle\ \BE^{\vee}(1)\, \frac{H_2}{1-\psi_1} \, \tau_0(\pt) \Big\rangle_{g-1,(1,d)}^{E_2 \times E_3, \text{red}} \,.
& =
\Big\langle\ 1 \ \Big|\ \BE^{\vee}(1) \frac{F}{1-\psi_1} \tau_0(\pt)\, \Big\rangle^{\p^1 \times E}_{g-2, (1,d)} \\
& =
\Big\langle\ 1 \ \Big|\ \BE^{\vee}(1) \frac{F}{1-\psi_1} \ \Big|\ \omega \ \Big\rangle^{\p^1 \times E}_{g-2, (1,d)}
\end{align*}
where $F$ is the class of a fiber over a point in $\p^1$,
and in the second step we used a further degeneration of the base $\p^1$ and Lemma~\ref{P1xE_Lemma}.
Using Lemma \ref{riogkegfr} and \eqref{assd} the claim follows now by summing up.
\end{proof}

\section{Formal series of quasi-modular forms} \label{Section_Formal_series_of_quasimodular_forms}
\subsection{Quasi-modular forms}
The ring of \emph{quasi-modular forms} is the free polynomial algebra
\[ \QMod = \BC[ C_2, C_4, C_6 ] \,, \]
where $C_{2k}$ are the Eisenstein series. The natural weight grading
\[ \QMod = \bigoplus_{m \geq 0} \QMod_m \]
is defined by assigning $C_{2k}$ weight $2k$.

For a quasi-modular form $f(q) = \sum_n a_n q^n$, let
\[ \nu(f) = \mathrm{inf} \{ \, n \, |\, a_n \neq 0\, \} \]
be the order of vanishing of $f$ at $q=0$.
If $f$ is a modular form, i.e. $f \in \BC[C_4, C_6]$, and $f$ is non-zero of weight $m$, then
\[ \nu(f) < \dim \Mod_m, \quad \text{ hence} \quad  \nu(f) \leq \frac{1}{12} m \,, \]
where $\Mod_m$ is the space of weight $m$ modular forms.
Similarly, one may ask if $\nu(f) < \dim \QMod_m$ also holds for
every non-zero quasi-modular form of weight $m$, see \cite{KK} for a discussion.
For us the following weaker bound proven by Saradha suffices:

\begin{lemma}[\cite{Saradha}, \cite{BP}] \label{Lemma:QMod_vanishing}
Let $f$ be a non-zero quasi-modular form of weight $2k$. Then 
\[ \nu(f) \leq \frac{1}{6} k (k+1) \,. \]
\end{lemma}
\begin{proof} The proof in \cite[Lemma 3]{Saradha}
also yields the stronger result stated here, as has been observed in \cite{BP}.
\end{proof}

\subsection{Formal series}
Let $u$ be a formal variable, and consider a power series
\[ \mathsf{F}(u,q) = \sum_{m \geq 0} f_m(q) u^m \]
in $u$ with coefficients $f_m(q) \in \QMod$. Let
$\big[ f_m(q) \big]_{q^n}$
denote the coefficient of $q^n$ in $f_m(q)$, and let
\[ \mathsf{F}_n(u) = \big[ \mathsf{F}(u,q) \big]_{q^n} = \sum_{m \geq 0} \big[ f_m(q) \big]_{q^n} u^m \]
be the series of $n$-th coefficients.

\begin{prop} \label{formel_prop}
Let $\sigma$ be an even integer, and let 
\[ \mathsf{F}(u,q) = \sum_{m \geq 0} f_m(q) u^m \]
be a formal power series in $u$ satisfying the following conditions:
\begin{enumerate}
 \item[(a)] $f_m(q) \in \QMod_{m+\sigma}$ for every $m$,
 \item[(b)] $\mathsf{F}_n(u)$ is the Laurent expansion of a rational function in $y$
 under the variable change $y = -e^{iu}$,
 \[ \mathsf{F}_n(u) = \sum_{r} c(n,r) y^r \,, \]
 \item[(c)] $c(n,r) = 0$ unless $r^2 \leq 4 n + 1$,
\item[(d)] $f_m(q) = 0$ for all $m \leq B(\sigma)$ where
\[ B(\sigma) = 2 \floor{ \sigma + 1 + \sqrt{2 \sigma^2 + 3 \sigma + 4} } \,. \]
\end{enumerate}
Then $\mathsf{F}(u,q) = 0$.
\end{prop}

\begin{proof}
Assume $F$ is non-zero.
Since $\sigma$ is even and all quasi-modular forms have even weight,
we have $f_m = 0$ unless $m$ is even.
Hence there exists an integer $b$
such that $f_m(q) = 0$ for all $m \leq 2b$, but $f_{2b+2}(q) \neq 0$.
Necessarily, $2 b \geq B(\sigma)$.

\noindent \emph{Claim.} $\mathsf{F}_n(u) = 0$ for $n < \frac{1}{4} b (b+2)$.

\noindent \emph{Proof of Claim.} By property (b) and (c) above, we may write
\[ \mathsf{F}_n(u) = \sum_{m \geq 0} a_m u^{2m} = \sum_{\ell = - \ell_{\text{max}}}^{\ell_{\text{max}}} c_\ell y^{\ell} \]
for coefficients $a_{m}, c_{\ell} \in \BC$ where $\ell_{\text{max}} = \floor{\sqrt{4n+1}}$.

Since $f_m = 0$ for all odd $m$, we find $\mathsf{F}_n(-u) = \mathsf{F}_n(u)$,
which yields the symmetry
$c_{\ell} = c_{-\ell}$.
In particular, we may also write
\[ \mathsf{F}_n(u) = \sum_{\ell = 0}^{\ell_{\text{max}}} b_{\ell} r^{2 \ell} \]
where
\[ r = y^{\frac{1}{2}} + y^{-\frac{1}{2}} = - 2 \sin\left( \frac{u}{2} \right) = - u + \frac{1}{24} u^3 + \ldots \,. \]
Since $r = - u + O(u^3)$ we obtain an invertible and upper-triangular relation
between the coefficients $\{ a_{\ell} \}_{\ell \geq 0}$ and $\{ b_{\ell} \}_{\ell \geq 0}$.
In particular, $a_{\ell} = 0$ for $\ell = 0, \dots, b$ implies
$b_{\ell} = 0$ for $\ell = 0, \dots, b$.
Since moreover $n < \frac{1}{4} b (b+2)$ implies $\ell_{\text{max}} \leq b$ we find
$b_{\ell} = 0$ for all $\ell$ and hence $\mathsf{F}_n = 0$ as claimed. \qed

We conclude the proof of Proposition \ref{formel_prop}. By the claim
the order of vanishing of $f_{2b+2}(q)$ at $q=0$ is at least $\frac{1}{4} b (b+2)$,
\[ \frac{1}{4} b (b+2) \leq \nu(f_{2b+2}) \,. \]
But by Lemma \ref{Lemma:QMod_vanishing} and the non-vanishing of $f_{2b+2}$,
\begin{equation} \nu(f_{2b+2}) \leq \frac{1}{6} (b+\sigma/2+1) (b+\sigma/2+2) \,, \label{use_of_vanishing_bound} \end{equation}
which is impossible since $2 b \geq B(\sigma)$.
\end{proof}

A crucial ingredient in the proof of Proposition \ref{formel_prop}
was the vanishing Lemma \ref{Lemma:QMod_vanishing} employed in equation \eqref{use_of_vanishing_bound}.
If we could prove
\begin{equation} \label{strong_bound} \nu(f) < \dim \QMod_{m} \end{equation}
for all non-zero quasi-modular forms of weight $m$, we could sharpen the bound in (d).
While we can't prove \eqref{strong_bound} for all $m$,
we have verified it for all $m \leq 250$.
This leads to the following partial strengthening of Proposition \ref{formel_prop}.
\begin{lemma} \label{formel_prop_strenthening} Assume $\sigma \leq 42$. Then Proposition \ref{formel_prop}
holds with property (d) replaced by
\begin{enumerate}
 \item[(d')] $f_m(q) = 0$ for $m \leq B'(\sigma)$, where $B'(\sigma)$ is
\[ 2 \cdot \mathrm{min} \Big\{ \widetilde{b} \in \BZ \ \Big|\ \frac{1}{4} b (b+2) > \dim \QMod_{\sigma + 2 b + 2} - 1 \text{ for all } b \geq \widetilde{b} \Big\} \,. \]
\end{enumerate}
\end{lemma}
\begin{proof} This follows by an argument identical to proof of Proposition \ref{formel_prop} except for the following steps:

If $b \leq 103$, then $2b+2+\sigma \leq 250$ and we use the bound \eqref{strong_bound} instead of \eqref{use_of_vanishing_bound}.
This leads to a contradiction by definition of $B'(\sigma)$.

If $b > 103$, then by assumption $f_m = 0$ for all $m \leq 208$. In particular, property (d)
of Proposition \ref{formel_prop} holds and we can apply Proposition \ref{formel_prop}.
\end{proof}

\noindent \textbf{Remarks.} (a) Since
\[ \dim \QMod_{2 \ell} = \frac{1}{12} \left( \ell^2 + 6 \ell + 12 \right) - c(\ell) \]
where $|c(\ell)| < 1$, the inequality
$ b (b+2) / 4 > \dim \QMod_{\sigma + 2 b + 2}$
holds for all $b$ sufficiently large. In particular, $B'(\sigma)$ defined above is well-defined and finite.
The first values are given in the following table:
\begin{center}
\begin{longtable}{|l | c | c | c | c | c | c | c | c | c | c | c | }
\hline
$\sigma$    & $<-2$ & $-2$ & $0$ & $2$ & $4$ & $6$ & $8$ & $10$ & $12$ & $14$ & $16$ \\
\hline
$B'(\sigma)$ & $-\infty$ & $2$ & $6$ & $10$ & $12$ & $14$ & $18$ & $20$ & $24$ & $26$ & $28$ \\
\hline
\end{longtable}
\end{center}
\vspace{-10pt}
In particular, for $\sigma < -2$ property (d') of Lemma~\ref{formel_prop_strenthening} is always
satisfied.

\vspace{3pt}
\noindent (b) We may obtain from Proposition \ref{formel_prop} a similar statement for odd $\sigma$ by integrating $\mathsf{F}$ formally with respect to $u$.

\vspace{3pt}
\noindent (c) The coefficient bound $r^2 \leq 4n + 1$ in Proposition \ref{formel_prop} (c) is the index~$1$ case
of the Fourier coefficient bound for weak Jacobi forms \cite{EZ}.
Surprisingly, the proof of Proposition \ref{formel_prop} fails
for higher index since these coefficient constraints become weaker,
while the growth of $\dim \QMod_{2 \ell}$ remains constant. The analog
of $B'(\sigma)$ is no longer well-defined.

\vspace{3pt}
\noindent (d) In applications below, the coefficient of $u^{2g + 2}$
in $\mathsf{F}(u,q)$ is a series of genus $g$ Gromov-Witten invariants of K3 surfaces.
For low $\sigma$, checking the vanishing of these coefficients in the range $2g + 2 \leq B(\sigma)$ is feasible.

\vspace{3pt}
\noindent (e) Proposition \ref{formel_prop} was motivated by
the proof of the Kudla modularity conjecture using formal series of Jacobi forms \cite{BR}.

\section{Genus induction} \label{Section:Genus_induction}
\subsection{Overview}
Let $S$ be an elliptic K3 surface with section, let $B$ and $F$
be the section and fiber class respectively, set
$\beta_h = B + h F$ where $h \geq 0$,
and let $\pt \in H^4(S,\BZ)$ be the class of a point.
Recall the generating series notation \eqref{Generating_Series_K3_Surfaces} for the surface $S$.
In this section we will prove the following evaluation:
\begin{thm} \label{mainthm_1} For all $m,n \geq 0$,
\[
\Big\langle\, \BE^{\vee}(1) \ \prod_{i=1}^{m} \frac{\pt}{1-\psi_i} \prod_{i=m+1}^{m+n} \frac{F}{1-\psi_i} \Big\rangle^{S}
= \frac{ ( \mathbf{G}(u,q) - 1 )^{m} \Theta(u,q)^{2n} }{ u^{2m+2n} \Theta(u,q)^2 \Delta(q)} .
\]
\end{thm}

\subsection{Formal series} \label{Subsection_proof_of_Theorem_1_Case2}
Theorem~\ref{mainthm_1} will follow from the following
evaluation and a degeneration argument.

\begin{thm} \label{thm_middle} $\displaystyle
\Big\langle\, \BE^{\vee}(1) \frac{\pt}{1 - \psi_1} \frac{F}{1 - \psi_2}\, \Big\rangle^{S}
 = 
 \frac{1}{u^4} \, \frac{\mathbf{G}(u,q) - 1}{\Delta(q)}
$
\end{thm}

Let $\omega \in H^2(\p^1)$ be the class of a point, and for $\gamma \in H^{\ast}(S)$ 
let
\[ \gamma \boxtimes \omega = \pi_1^{\ast}(\gamma) \cup \pi_2^{\ast}(\omega) \in H^{\ast}(S \times \p^1) \]
where $\pi_i$ is the projection of $S \times \p^1$ to the $i$th factor. Define the formal series
\begin{equation} \label{Series_F}
\mathsf{F}(u,q)
=
\Delta(q) \cdot \sum_{g\in \BZ} \sum_{h \geq 0}
\big\langle \tau_0(\pt \boxtimes \omega) \tau_0(F \boxtimes \omega)^3 \big\rangle^{S \times \p^1, \bullet}_{g, (\beta_h,2)} u^{2g+2} q^{h-1}
\end{equation}
where the bracket on the right hand side denotes disconnected absolute Gromov-Witten invariants of $S \times \p^1$.

\begin{lemma} \label{wer} With $D = q \frac{d}{dq}$,
\[ \mathsf{F}(u,q) =
u^4 \Delta(q) \Big\langle\, \BE^{\vee}(1) \, \frac{\pt}{1 - \psi_1} \, \frac{F}{1 - \psi_2}\, \Big\rangle^{S}
+ 1 + \Theta(u,q) \cdot D \Theta(u,q) \]
\end{lemma}
\begin{proof}
By Proposition \ref{vanishing1} the contribution from disconnected curves to
\begin{equation} \label{bbb}
\sum_{g\in \BZ} \sum_{h \geq 0} \big\langle \tau_0(\pt \otimes \omega) \tau_0(F \boxtimes \omega)^3 \big\rangle^{S \times \p^1, \bullet}_{g, (\beta_h,2)} u^{2g+2} q^{h-1}
\end{equation}
is $\frac{1}{\Delta(q)}$.\footnote{If the curve is disconnected it must have precisely two components of degree $1$ over $\p^1$ each.
Moreover, one component carries the insertion $\pt \otimes \omega$ and contributes $1$,
the other carries all the insertions $F \boxtimes \omega$ and contributes $\Delta(q)^{-1}$.}
For the contribution from connected curves we apply the localization formula,
specializing $\pt \boxtimes \omega$ and one $F \boxtimes \omega$ insertion to the fiber over $\infty$,
and the other insertions to the fiber over $0 \in \p^1$.
We find \eqref{bbb} equals
\[
\frac{1}{\Delta(q)} +
u^4 \Big\langle \BE^{\vee}(1) \frac{\pt}{1 - \psi_1} \frac{F}{1 - \psi_2} \Big\rangle^{S}
+ u^4 \Big\langle \BE^{\vee}(1) \frac{F}{1 - \psi_1} \frac{F}{1 - \psi_2} \tau_0(\pt) \Big\rangle^{S}.
\]
We evaluate the third term
by degenerating $S$ to a union of $S$ with four bubbles of $\p^1 \times E$,
\[ S \leadsto S \cup (\p^1 \times E) \cup \ldots \cup (\p^1 \times E) \]
where the first three copies of $\p^1 \times E$ receive a single insertion each.
By \eqref{assd}, \eqref{KKV} and Lemma \ref{P1xE_Lemma},
\begin{multline*}
u^4 \Big\langle\, \BE^{\vee}(1) \cdot \frac{F}{1 - \psi_1} \cdot \frac{F}{1 - \psi_2}\, \tau_0(\pt) \Big\rangle^{S}
=
u^4 \blangle \BE^{\vee}(1) \big| 1 \brangle^S \\
\cdot \Big( \blangle \omega \big| \, F / (1-\psi) \, \big| 1 \brangle^{\p^1 \times E} \Big)^2
\cdot \blangle \omega \big| \hodge \tau_0(\pt) \big| 1 \brangle^{\p^1 \times E}
\cdot \blangle \omega \big| \hodge \brangle^{\p^1 \times E} \\
= \frac{\Theta(u,q) D \Theta(u,q)}{\Delta(q)} \,. \qedhere
\end{multline*}
\end{proof}

\begin{prop} \label{rgrgreg}
The series $\mathsf{F}(u,q)$ satisfies properties \textup{(a)}, \textup{(b)}, \textup{(c)} of Proposition~\ref{formel_prop} with $\sigma=0$.
\end{prop}

\begin{proof}
\noindent \textbf{Property (a).}
For $m \in \BZ$ let
\[ f_m(q) = \big[ \, \mathsf{F}(u,q) \, \big]_{u^m} \]
be the coefficient of $u^m$ in $\mathsf{F}(u,q)$.
For odd $m$, $f_m(q)$ vanishes. For even $m$ we have by Lemma~\ref{wer}
\[ f_m(q) = \Delta(q) \sum_{h \geq 0} \Big\langle \BE^{\vee} \, \frac{\pt}{1 - \psi_1} \, \frac{F}{1 - \psi_2} \Big\rangle^{S}_{g, \beta_h} q^{h-1}
+ \delta_{m0} + \big[ \Theta \cdot D \Theta \big]_{u^m} \]
where $m = 2g+2$. By the refinement \cite[Theorem 9]{BOPY} of the quasi-modularity result proven in \cite{MPT},
the first term on the right hand side is a quasi-modular form of weight $2g+2$.
By direct verification the last two terms are also quasi-modular of weight $m$. Hence
\[ f_m(q) \in \QMod_m \,.  \]
This verifies property (a).

\vspace{6pt}
\noindent \textbf{Property (b).}
By an argument parallel to the proof of \cite[Proposition 5]{K3xE}
the GW/Pairs correspondence \cite{PaPix1,PaPix2}
holds for absolute disconnected Gro\-mov-Witten invariants of $S \times \p^1$ in class $(\beta_h,d)$.
In particular, the coefficient of every $q^{h-1}$ in \eqref{bbb}
is the Laurent expansion of a rational function in $y$ under the variable transformation $y = - e^{iu}$.
This implies the claim for $\mathsf{F}(u,q)$.

\vspace{6pt}
\noindent \textbf{Property (c).} For each $h \geq 0$ consider the Laurent expansion
\begin{equation}
\sum_{g} u^{2g+2}
\big\langle \tau_0(\pt \boxtimes \omega) \tau_0(F \boxtimes \omega)^3 \big\rangle^{S \times \p^1, \bullet}_{g, (\beta_h,2)}
 = \sum_{r \in \BZ} c(h,r) y^r \label{u-expansion} \end{equation}
of the rational function in $y = - e^{iu}$.
By the GW/Pairs correspondence\footnote{See Property (b).} we have
\[ c(h,r) = \big\langle \tau_0(\pt \boxtimes \omega) \tau_0(F \boxtimes \omega)^3 \rangle^{S \times \p^1, \text{Pairs}}_{(\beta_h,2), r+2} \,, \]
where the right hand side are reduced
stable pairs invariants of $S \times \p^1$ in class $(\beta_h,2)$ with Euler characteristic $r+2$.

We will prove the vanishing of $c(h,r)$ for $r^2 > 4h+1$ in three steps.

\vspace{4pt}
\noindent \textbf{Step 1.} $c(h,r) = 0$ for $r < -\sqrt{4 h + 1}$.

By deformation invariance we may assume $\beta_h$ is irreducible.
Let $F_i, i=1,2,3$ be generic disjoint smooth submanifolds in class $F$, let $x_1, x_2,x_3 \in \p^1$ be distinct points,
and let $P \in S \times \p^1$ be a generic point.
Let
\[ \mathsf{P}(\beta_h, n) \]
denote the moduli space of stable pairs in $S \times \p^1$
of class $(\beta_h,2)$ Euler characteristic $n$ and whose underlying support curve is incident to $F_i \times x_i$ for $i \in \{ 1,2,3 \}$ and to the point $P$.

We claim $\mathsf{P}(\beta_h, n)$ is empty if $n < 2 -\sqrt{4h+1}$.

Indeed, let $[ \CO_X \to \CF ] \in \mathsf{P}(\beta_h, n)$ with underlying support curve $C$.

If $C$ is disconnected and incident to $P$ and $x_i \times F_i, i=1,2,3$, then
$C$ is a disjoint union of two copies of $\p^1$. Hence,
\[ n = \chi(\CF) \geq \chi(\CO_C) \geq 2 \,. \]

If $C$ is connected, the incidence conditions imply that $C$ is irreducible and reduced.
Then by Theorem~\ref{thm_CK} the arithmetic genus $g = g_a(C) = 1 - \chi(\CO_C)$ satisfies
\[ h \geq g + \alpha( g - \alpha - 1 ) \]
where $\alpha = \floor{g/2}$ which implies $n = \chi(\CF) \geq \chi(\CO_C) \geq 2 - \sqrt{4h+1}$.

Since $n = r+2$, Step 1 is complete.

\vspace{4pt}
\noindent \textbf{Step 2.} There exist an intger $N \geq 0$ and $n_{g,h} \in \BQ$ such that
\[ \sum_r c(h,r) y^r = \sum_{g=-N}^N n_{g,h} (y^{1/2} + y^{-1/2})^{2g+2} \,. \]

\noindent \emph{Proof.} By Lemma~\ref{wer} and the expansion of $\Theta(u,q)$ in $y = -e^{iu}$ it is enough to show that
for all $h$
\begin{equation}
\sum_{g \geq 0} \Big\langle\, \BE^{\vee}(1) \, \frac{\pt}{1 - \psi_1} \, \frac{F}{1 - \psi_2}\, \Big\rangle^{S}_{g,\beta_h} u^{2g+2}
 = \sum_{g=-N}^{N} n_{g,h} (y^{\frac{1}{2}} + y^{-\frac{1}{2}})^{2g+2}
\label{301}
\end{equation}
for some $N$ and some $n_{g,h}$ under the variable change $y = -e^{iu}$.
For this, we will relate the left hand side to the Gromov-Witten invariants of $S \times E$, where $E$ is an elliptic curve.

By degenerating two $\p^1 \times E$-bubbles off from $S$, and by the Katz-Klemm-Vafa formula and \eqref{assd}, we have
\[
u^4 \Big\langle\, \BE^{\vee}(1) \, \frac{\pt}{1 - \psi_1} \, \frac{F}{1 - \psi_2}\, \Big\rangle^{S}
=
\frac{u^2}{\Delta(q)} \Big\langle\, \omega \, \Big| \, \BE^{\vee}(1) \, \frac{\pt}{1 - \psi_1} \Big\rangle^{\p^1 \times E} \,.
\]
On the other side, let $\omega \in H^2(E,\BZ)$ be the class of a point and let
\[
\big\langle\tau_0(F \boxtimes \omega) \big\rangle^{S \times E, \bullet}
=
\sum_h \sum_{g} \blangle \tau_0(F \boxtimes \omega) \brangle^{S \times E, \bullet}_{g, (\beta_h, 1)} u^{2g-2} q^{h-1}
\]
be the generating series of disconnected Gromov-Witten invariants of $S \times E$.
By degenerating $E$ to a nodal curve and resolving we have
\begin{equation} \label{302}
\big\langle\tau_0(F \boxtimes \omega) \big\rangle^{S \times E, \bullet}
= \sum_{\gamma} \big\langle \gamma, \gamma^{\vee} \Big| \tau_0(F \boxtimes \omega) \big\rangle^{S \times \p^1 / \{ 0, \infty \}, \bullet}
\end{equation}
where $\gamma$ runs over a basis of $H^{\ast}(S, \BQ)$ with $\gamma^{\vee}$ the dual basis,
and we have written $\gamma$ for the weighted partition $(1,\gamma)$.
Degenerating the base $\p^1$ to two copies of $\p^1$ with the non-relative point specializing
to one, and the relative marked points specializing to the other, the right hand side of \eqref{302} is
\begin{equation} \label{300}
\sum_{\gamma} \big\langle \gamma, \gamma^{\vee}, F \big\rangle^{S \times \p^1 / \{ 0, 1, \infty\} , \bullet}
+ 24 \big\langle \, \pt \, \Big|\, \tau_0(F \boxtimes \omega) \big\rangle^{S \times \p^1, \bullet} \,.
\end{equation}
By arguments parallel\footnote{We may also use Proposition~\ref{Proposition_take_out_div} below to reduce to Proposition~\ref{vanishing1}.}
to the proof of Proposition~\ref{vanishing1},
only genus $0$ curves contribute to the first term in \eqref{300}. Hence,
\[ \sum_{\gamma} \big\langle \gamma, \gamma^{\vee}, F \big\rangle^{S \times \p^1 / \{ 0, 1, \infty\} , \bullet} = g(q) \]
for some power series $g(q)$ independent of $u$.
By using the Katz-Klemm-Vafa formula for the disconnected,
and the localization formula for the connected part, the second term of \eqref{300} is
\[
\frac{24}{\Theta(u,q)^2 \Delta(q)} +
24 u^2 \Big\langle \BE^{\vee}(1) \frac{\pt}{1-\psi_1} \Big\rangle^S
\]
which by a further degeneration $S \leadsto S \cup (\p^1 \times E)$ is
\[
\frac{24}{\Theta(u,q)^2 \Delta(q)} \left( 1 + u^2 \Big\langle\, \omega \, \Big| \, \BE^{\vee}(1) \, \frac{\pt}{1 - \psi_1} \Big\rangle^{\p^1 \times E} \right)
\]
Combining everything, we have therefore
\begin{multline*}
u^2 \Big\langle\, \omega \, \Big| \, \BE^{\vee}(1) \, \frac{\pt}{1 - \psi_1} \Big\rangle^{\p^1 \times E}
=
\frac{\Theta(u,q)^2 \Delta(q)}{24} \left( \Big\langle\tau_0(F \boxtimes \omega) \Big\rangle^{S \times E, \bullet} - g(q) \right) - 1 
\end{multline*}

By the GW/Pairs correspondence for $S \times E$ in primitive classes, see \cite[Proposition 5]{K3xE},
the series $\langle\tau_0(F \boxtimes \omega) \rangle^{S \times E, \bullet}$
equals a series of reduced stable pair invariants for $X$ under $y = -e^{iu}$.
By \cite[Thm. 1]{ReducedSP} these can be evaluated by the Behrend function weighted Euler characteristic
of the quotient of the moduli space of stable pairs by the translation action by the elliptic curve.
The result now follows from \cite[Thm. 2]{ReducedSP} or
alternatively \cite[Section 4, Appendix]{PTBPS}
(since the classes $(\beta_h, 1)$ are reduced in the sense of \cite{PTBPS}).

\vspace{4pt}
\noindent \textbf{Step 3.} $c(h,r) = 0$ for $r > \sqrt{4 h + 1}$.

\noindent \emph{Proof.} For every $h$, consider the rational function
\[ f(y) = \sum_r c(h,r) y^r = \sum_{g=-N}^N n_{g,h} (y^{1/2} + y^{-1/2})^{2g+2} \,. \]
Substituting $y = -e^{iu}$ and taking the Laurent expansion around $u=0$,
we obtain the equality of formal Laurent series
\[ f(-e^{iu})
= \sum_{g} u^{2g+2} \big\langle \tau_0(\pt \boxtimes \omega) \tau_0(F \boxtimes \omega)^3 \big\rangle^{S \times \p^1, \bullet}_{g, (\beta_h,2)} \,.
\]

By considering a generic K3 surface and a generic choice of cycles representing the incidence conditions,
a direct check shows
\[ \big\langle \tau_0(\pt \boxtimes \omega) \tau_0(F \boxtimes \omega)^3 \big\rangle^{S \times \p^1, \bullet}_{g, (\beta_h,2)} = 0 \]
for $g \leq -2$. Hence, $f(-e^{iu})$ is a power series in $u$:
\[ f(-e^{iu}) = a_0 + a_2 u^2 + a_4 u^4 + \ \ldots , \quad \quad a_i \in \BC \,. \]
Since
\[ (y^{1/2} + y^{-1/2})^{2g+2} = u^{2g+2} + O(u^{2g+4}) \]
this shows $n_{g,h} = 0$ for $g \leq -2$.
Hence, $f(y)$ is a finite Laurent \emph{polynomial} in $y$,
\[ f(y) = \sum_{r=-M}^{M} c(h,r) y^r = \sum_{g=-1}^{N} n_{g,h} (y^{1/2} + y^{-1/2})^{2g+2} \,. \]
Since $f$ is symmetric under $y \mapsto y^{-1}$, we conclude
\[ c(h,r) = c(h,-r) \,. \]
The claim of Step 3 follows now from Step 1 above.

The proof of Property (c) for $\mathsf{F}(u,q)$ is now complete.
\end{proof}

\subsection{Proof of Theorem~\ref{thm_middle}} \label{Subsection_Proof_Thm_middle}
Let $\mathsf{F}(u,q)$ be the formel series defined in \eqref{Series_F}.
By Lemma~\ref{wer} it is enough to show
\begin{equation} \mathsf{F}(u,q) = \mathbf{G}(u,q) + \Theta(u,q) \cdot D \Theta(u,q) \,. \label{utut} \end{equation}
By Proposition \ref{rgrgreg} the left hand side satisfies the properties (a)-(c) of Proposition~\ref{rgrgreg}.
Since we may rewrite
\[ \mathbf{G}(u,q) = \frac{1}{12} \varphi_{0,1}(z,\tau) - 2 C_2(q) \varphi_{-2,1}(z,\tau) \]
where $q = e^{2 \pi i \tau}, u = 2 \pi z$ and $\varphi_{0,1}, \varphi_{-2,1}$ are
the weak Jacobi forms of index~$1$ defined in \cite[Section 9]{EZ},
properties (a)-(c) of Proposition \ref{rgrgreg} hold for $\mathbf{G}$,
and similarly for $\Theta \cdot D \Theta$.\footnote{For example, see \cite[page 105]{EZ} for the crucial coefficient bound.}

Hence by Proposition~\ref{formel_prop} resp. Lemma \ref{formel_prop_strenthening}, we need to check \eqref{utut}
only for the coefficients of $u^m$ where $m \leq 6$,
or equivalently since $m=2g+2$ for genera $0 \leq g \leq 2$.
For this, we may reduce by Lemma \eqref{wer} to
Gromov-Witten invariants of a K3 surface with only fiber and point insertions.
These can be computed for fixed genus by a degeneration argument, see \cite{MPT} or Appendix \ref{Appendix_K3}. \qed

\subsection{Proof of Theorem \ref{mainthm_1}} \label{Subsection_Proof_of_mainthm}
Consider the degeneration of $S$ to the union of $S$ with $m+n+1$ bubbles of $\p^1 \times E$,
\[ S \leadsto S \cup \underbrace{(\p^1 \times E) \cup \ldots \cup (\p^1 \times E)}_{m+n+1} \,. \]
Applying the degeneration formula to
\[  \Big\langle \BE^{\vee}(1)\ \prod_{i=1}^{m} \frac{\pt}{1-\psi_i} \prod_{i=1}^{n} \frac{F}{1-\psi_i} \Big\rangle^{S} \]
with the first $m+n$ copies of $\p^1 \times E$ receiving a single insertion each,
yields by Lemma \ref{P1xE_Lemma}
\begin{multline} \label{pppp}
\Big\langle\ \BE^{\vee}(1)\ \Big| \ 1 \  \Big\rangle^{S}
 \left( \Big\langle\ \omega \ \Big| \ \BE^{\vee}(1) \frac{F}{1-\psi_1} \ \Big| \ 1 \ \Big\rangle^{\p^1 \times E} \right)^{n} \\  
 \cdot \left( \Big\langle\ \omega \ \Big| \ \BE^{\vee}(1) \frac{\pt}{1-\psi_1} \ \Big| \ 1 \ \Big\rangle^{\p^1 \times E} \right)^{m} \,.
\end{multline}
The first term on the right is the Katz-Klemm-Vafa formula \eqref{KKV},
the second term is determined by \eqref{assd}. By solving for the 
third term in case $m=n=1$ using the result of Theorem~\ref{thm_middle} we find
\[ \Big\langle\ \omega \ \Big| \ \BE^{\vee}(1) \frac{\pt}{1-\psi_1} \ \Big| \ 1 \ \Big\rangle^{\p^1 \times E} = \frac{\mathbf{G}(u,q) - 1}{u^2} \,. \]
Inserting everything back into \eqref{pppp} completes the proof. \qed

\subsection{Further invariants}
Theoretically we could use the formal method used above
to evaluate also other Gromov-Witten invariants of $S \times \p^1$ in classes $(\beta_h,2)$.
For example consider the relative invariants
\begin{equation} \label{ffffg}
\Big\langle (1,F)^2 , D(F) , (1,F)^2 \Big\rangle^{S \times \p^1, \bullet}_{g,(\beta_h,2)},
\end{equation}
which, under the GW/Hilb correspondence,
count rational curves in $\Hilb^2(S)$
incident to $2$ fibers of a Lagrangian fibration $\Hilb^2(S) \to \p^2$ \cite{HilbK3}.
The appropriate generating series associated to \eqref{ffffg} satisfies almost all
conditions needed for Proposition \ref{formel_prop}. (Showing property (c) for $r > \sqrt{4n+1}$ requires a BPS expansion
parallel to the one used in Step 2 of the proof of Proposition \ref{rgrgreg} for which
we do not have a full argument at the moment.)
The modular weight $\sigma$ takes the lowest possible value namely $\sigma = -2$.
Therefore we expect \eqref{ffffg}
to be determined by formal properties and the evaluation in genus $0$ alone (which is the Yau-Zaslow formula).

Similarly the space of quasi-Jacobi forms of index~$1$ and weight $-2$ has dimension~$1$, and is spanned by
\[ \varphi_{-2,1}(z,\tau) = \Theta(u,q)^2 \,. \]
By comparision and without any further calculation we find that $\frac{\Theta(u,q)^2}{\Delta(q)}$ is
the generating series for \eqref{ffffg}.
By localization and degeneration the Gromov-Witten invariants of $S \times \p^1$ reduce to
linear Hodge integrals on the K3 surface.
The reasoning above provides some explanation of the ubiquity
of Jacobi forms and particularly of $\Theta(u,q)$
in the enumerative geometry of K3 surfaces.


\section{Relative invariants of $\text{K3} \times \p^1$} \label{Section:Relative_Invariants_of_P1K3}
\subsection{Overview}
The main objective of this section is to prove the GW/Hilb correspondence in degree $2$ (Theorem \ref{GWHilb_thm}).
In \cite{HilbK3} the full genus~$0$ three point theory of $\Hilb^2(S)$ for primitive classes
has been determined by calculating first five basic cases,
and then applying the WDVV equations repeatedly to solve for all other invariants.
Here we follow a similar approach. In Section~\ref{Section_Relations} we first show the WDVV equations
for genus $0$ invariants on $\Hilb^d(S)$
is compatible with a corresponding set of equations for the relative invariants of $S \times \p^1$
obtained from the degeneration formula.
In Section \ref{Section_Proof_Of_MainThm_1b}, independently from the rest, we prove Theorem \ref{mainthm_1b}.
In Section~\ref{Section_Special_cases_in_degree_2} we use a combination of standard methods and Theorem~\ref{mainthm_1}
to calculate the same five basic series as in the $\Hilb^d(S)$ case.
Since these series match those on the $\Hilb^d(S)$ side this completes the proof.
Finally in Section~\ref{Section_Proof_of_SxE_Theorem} we prove Theorem~\ref{Theorem_K3xE}.

Throughout the section we will repeatedly use the localization
and the degeneration formula, see for example \cite{GP, GV, FPM} and \cite{Junli1, Junli2, MP}.

\subsection{Relations} \label{Section_Relations}
Let $S$ be a K3 surface.
Let $\{ \gamma_i \}_i$ be a fixed basis of $H^{\ast}(S)$.
We identify a partition $\mu = \{ (\mu_j, \gamma_{i_j}) \}$ weighted by the basis $\{ \gamma_i \}$
with the class
\[ \frac{1}{\Fz(\mu)} \prod_j \Fp_{-\mu_j}(\gamma_{i_j}) v_{\varnothing} \in H^{\ast}(\Hilb^{|\mu|}(S)) \]
on the Hilbert scheme, where $\Fz(\mu) = |\Aut(\mu)| \prod_{i}\mu_i$. Let also
$\deg(\mu)$ denote the complex cohomological degree of $\mu$ in $\Hilb^d(S)$,
\[ \mu \in H^{2 \deg(\mu)}(\Hilb^d(S)) \,. \]

Since $\{ \gamma_i \}$-weighted partitions of size $d$ form a basis for the 
cohomology of $\Hilb^d(S)$, the cup product $\mu \cup \nu$ of cohomology weighted partitions $\mu, \nu$
can be uniquely expressed as a formal linear combination of weighted partitions:
\[ \mu \cup \nu = \sum_{\lambda} c_{\mu \nu}^{\lambda} \lambda \]
where the sum runs over all weighted partitions of size $|\mu|$ and $c_{\lambda} \in \BQ$ are coefficients.
When $\mu$ or $\nu$ are divisor classes on $\Hilb^d(S)$, explicit formulas for $\mu \cup \nu$ are surveyed in \cite{Lehn}.

Let $\mu, \nu, \rho$ be cohomology weighted partitions of size $d$,
and let $\beta \in H_2(S,\BZ)$ be a curve class.
We will require the modified bracket
\[
\big\langle \, \mu ,\nu , \rho \, \big\rangle^{S \times \p^1, \star}_{\beta}
 =
 (-iu)^{l(\mu) + l(\nu) + l(\rho) - d} \sum_{g \in \BZ} \big\langle \, \mu ,\nu , \rho \, \big\rangle^{S \times \p^1,\bullet}_{g, (\beta,d)} u^{2g-2} .
\]
where the bracket on the right hand side denote
disconnected Gromov-Witten invariants of $S \times \p^1 / \{ 0,1,\infty\}$
with relative insertions $\mu,\nu,\rho$.
Since the degree $d$ is determined by the partition,
it is omitted in the notation from the left hand side.
When the entries $\mu, \nu, \rho$ are formal linear combination of cohomology weighted partitions,
the bracket $\langle \mu, \nu, \rho \rangle^{S \times \p^1, \star}$ is defined by multilinearity

\begin{prop} \label{Proposition_WDVV_analog}
Let $\lambda_1, \dots, \lambda_4$ be cohomology weighted partitions of size $d$ weighted by the fixed basis $\{ \gamma_i \}$,
such that $\sum_i \deg(\lambda_i) = 2d+1$. Then
\begin{multline*}
\blangle \lambda_1, \lambda_2, \lambda_3 \cup \lambda_4 \brangle_{\beta}^{S \times \p^1, \star}
+ \blangle \lambda_1 \cup \lambda_2, \lambda_3, \lambda_4 \brangle_{\beta}^{S \times \p^1, \star} \\
=
\blangle \lambda_1, \lambda_4, \lambda_2 \cup \lambda_3 \brangle_{\beta}^{S \times \p^1, \star}
+ \blangle \lambda_1 \cup \lambda_4, \lambda_2, \lambda_3 \brangle_{\beta}^{S \times \p^1, \star} \,.
\end{multline*}
\end{prop}

\begin{proof}
Consider Gromov-Witten invariants of $S \times \p^1$ relative to fibers over the points $0,1, \infty, t \in \p^1$,
\begin{equation}
\label{4ptseries} (-iu)^{-2d + \sum_i l(\lambda_i)} \sum_g
\blangle \lambda_1, \lambda_2, \lambda_3, \lambda_4 \brangle_{g,(\beta,d)}^{(S \times \p^1) / \{ 0, 1, \infty, t\}, \bullet}
u^{2g-2} \,.
\end{equation}
Consider the degeneration of $S \times \p^1$ obtained by degenerating the base $\p^1$ to a union of two copies of $\p^1$,
\[ S \times \p^1 \leadsto (S \times \p^1) \cup (S \times \p^1). \]
We assume the fibers over $0,1$ specialize to the first and the fibers over $t, \infty$ specialize to
the second component respectively.
We will apply the degeneration formula to \eqref{4ptseries}.
Since the reduced class breaks into a product of a reduced class and an ordinary virtual class,
we must have either $\beta_1 = 0$ or $\beta_2 = 0$ in the splitting $\beta = \beta_1 + \beta_2$
of the curve class.
The result of the degeneration formula is
\begin{multline} \label{4ptseries_broken}
\sum_{g_1, g_2} \sum_{\beta=\beta_1 + \beta_2} \sum_{\eta}
\langle \lambda_1, \lambda_2, \eta \brangle^{S \times \p^1, \bullet}_{g_1, \beta_1} \\
\cdot \Fz(\eta) \blangle \eta^{\vee}, \lambda_3, \lambda_4 \brangle_{g_2,\beta_2}^{S \times \p^1, \bullet}
(-iu)^{-2d + \sum_i l(\lambda_i)} u^{2(g_1 + g_2 + l(\eta) - 1)-2}
\end{multline}
where $g_1, g_2$ run over all integers,
we have $\beta_1 = 0$ or $\beta_2 = 0$,
$\eta$ runs over all $\{ \gamma_i \}$-weighted partitions of size $d$,
and $\eta^{\vee}$ is the dual partition\footnote{If $\eta = \{ (\eta_i, \gamma_{s_i}) \}$, then $\eta^{\vee} = \{ (\eta_i, \gamma_{s_i}^{\vee}) \}$
where $\{ \gamma_i^{\vee} \}$ is the basis dual to $\{ \gamma_i \}$ with respect to the intersection pairing on $H^{\ast}(S,\BQ)$.}.
Above, we also have used the genus glueing relation
\[ g = g_1 + g_2 + l(\eta) - 1 \,, \]
and have followed the notation (explained in Section \ref{Section_The_bracket_notation}) that we use a reduced class whenever
the K3 factor of the curve class is non-zero, and the usual virtual class otherwise.

Consider the basis of $H^{\ast}(\Hilb^d(S))$ defined by the set 
of all $\{ \gamma_i \}$-weighted partitions $\eta$ of size $d$.
The corresponding dual basis with respect to the intersection pairing on $H^{\ast}(\Hilb^d(S))$ is
$\{ (-1)^{d + l(\eta)} \Fz(\eta) \eta^{\vee} \}$.
Hence for every $\alpha \in H^{\ast}(\Hilb^d(S))$,
\[ \alpha = \sum_{\eta} (-1)^{d + l(\eta)} \Fz(\eta) \langle \alpha, \eta^{\vee} \rangle \eta \,. \]
We will also require the following evaluation of (non-reduced) relative invariants of $S \times \p^1$ in class $(0,d)$,
\begin{equation}
\label{non-red-cor} (-iu)^{-d + \sum_i l(\lambda_i)} \sum_g \langle \lambda_1, \lambda_2, \lambda_3 \rangle^{S \times \p^1, \bullet}_{g, (0,d)} u^{2g-2}
 = (-1)^d \int_{\Hilb^d(S)} \lambda_1 \cup \lambda_2 \cup \lambda_3 \,.
\end{equation}
for all weighted partitions $\lambda_1, \lambda_2, \lambda_3$. Equality \eqref{non-red-cor}
follows directly from the corresponding local case, see \cite[Section 4.3]{OPHilb}.

Putting everything together, \eqref{4ptseries_broken} and hence \eqref{4ptseries} are equal to
the left hand side of Proposition \ref{Proposition_WDVV_analog}, namely
\[
\blangle \lambda_1, \lambda_2, \lambda_3 \cup \lambda_4 \brangle^{S \times \p^1, \star}_{\beta}
+ \blangle \lambda_1 \cup \lambda_2, \lambda_3, \lambda_4 \brangle^{S \times \p^1, \star}_{\beta} \,.
\]
Since by a parallel argument (with $0,t$ specializing to the first, and $1,\infty$ specializing to the second component)
we also find \eqref{4ptseries} to equal the right hand side of Proposition \ref{Proposition_WDVV_analog}.
\end{proof}
From Proposition \ref{Proposition_WDVV_analog} and \cite[Appendix A]{HilbK3} we obtain the following.
\begin{cor} Under the correspondence of Conjecture \ref{GW/Hilb_correspondence},
the reduced WDVV equation on $\Hilb^d(S)$ corresponds to
the degeneration relations of Proposition \ref{Proposition_WDVV_analog}
\end{cor}

For weighted partitions $\mu, \nu, \rho$ of size $d$, let
\[
\big\langle \, \mu ,\nu , \rho \, \big\rangle^{S \times \p^1,\bullet}_{\beta}
= 
u^{l(\mu) + l(\nu) + l(\rho) - d}
\sum_{g \in \BZ}
\langle \mu , \nu , \rho \brangle^{S \times \p^1/ \{ 0,1,\infty \}, \bullet}_{g, (\beta,d)}
u^{2g - 2}
\]
\begin{prop} \label{Proposition_take_out_div} For all $\gamma, \gamma' \in H^2(S, \BQ)$ and all weighted partitions $\mu, \nu$ of size $d$,
\begin{align*}
\langle \beta, \gamma' \rangle \cdot \big\langle \, \mu ,\nu , D(\gamma) \, \big\rangle_{\beta}^{S \times \p^1,\bullet}
& = \langle \beta, \gamma \rangle \cdot \big\langle \, \mu ,\nu , D(\gamma') \, \big\rangle_{\beta}^{S \times \p^1,\bullet} \\
\langle \beta, \gamma \rangle \cdot \big\langle \, \mu ,\nu , (2, \e)(1, \e)^{d-2} \, \big\rangle_{\beta}^{S \times \p^1,\bullet}
& = \frac{d}{du} \, \big\langle \, \mu ,\nu , D(\gamma) \, \big\rangle_{\beta}^{S \times \p^1,\bullet}
\end{align*}
\end{prop}
\begin{proof} By a rubber calculus argument, see for example \cite[Prop. 4.3]{M} or \cite{MP}.
\end{proof}

%

\subsection{Elliptic K3 surfaces}
In the remainder of Section~\ref{Section:Relative_Invariants_of_P1K3}
let $S$ be an elliptically fibered K3 surface with section,
let $B$ and $F$ be the section and fiber class respectively, and let
$\beta_h = B +hF$ for all $h \geq 0$.

For $H^{\ast}(S)$-weighted partitions $\mu, \nu, \rho$ of size $d$, we set
\begin{multline} \label{rel_gen_ser}
\big\langle \, \mu ,\nu , \rho \, \big\rangle^{S \times \p^1,\bullet}
= \ 
u^{l(\mu) + l(\nu) + l(\rho) - d} \sum_{g \in \BZ} \sum_{h \geq 0}
\langle \mu , \nu , \rho \brangle^{S \times \p^1/ \{ 0,1,\infty \}, \bullet}_{g, (\beta_h,d)}
u^{2g - 2} q^{h-1}
\end{multline}
for the generating series of disconnected Gromov-Witten invariants of $S \times \p^1 / \{ 0,1, \infty\}$,
and the same except without $\bullet$ for connected invariants.

\subsection{Proof of Theorem \ref{mainthm_1b}} \label{Section_Proof_Of_MainThm_1b}
Let $\mu_{m,n}, \nu_{m,n}$ be the weighted partitions defined in \eqref{Rel_Conditions}.
For the proof we will drop the subscript $n$ and simply write $\mu_m = \mu_{m,n}$, etc.
Let also $\rho_m = \{ (1,F) (1, \e)^{m+n-1} \}$.

Let $n>0$ first. By a degeneration argument the \emph{connected} invariants satisfy
\[
\big\langle \mu_m , \nu_m, \rho_m \big\rangle^{S \times \p^1/ \{ 0,1,\infty \}}_{g,(m+n,\beta_h)}
= \frac{1}{n!} \big\langle \mu_m \big| \tau_{0}(F \boxtimes \omega)^{n+1} \big\rangle^{S \times \p^1/ \{ 0 \} }_{g,(m+n,\beta_h)}.
\]
Applying the localization formula and Theorem~\ref{mainthm_1} yields
\[
\sum_{g, h} \big\langle \mu_m , \nu_m, \rho_m \big\rangle^{S \times \p^1/ \{ 0,1,\infty \}}_{g,(m+n,\beta_h)} q^{h-1} u^{2(m+n) + 2g-2}
=\frac{1}{n!^2 m!} \frac{(\mathbf{G}-1)^{m} \Theta^{2n}}{\Theta^2 \Delta}\,.
\]

To obtain the disconnected invariants, let
\[ f : C \to S \times \p^1 \]
be a possibly disconnected relative stable map
incident to (cycles representing) $\mu_m, \nu_m, \rho_m$ over $0,1, \infty$ respectively.
There is a single connected component $C_0$ of $C$ such that
the restriction $f|C_0$ maps in class $(n+k,\beta_h)$ for some $k \geq 0$ 
and is incident to $\mu_k, \nu_k, \rho_k$.
By the incidence conditions, the restriction of $f$ to every other component is
an isomorphism onto a rational line $\p^1 \times P$ where $P$
is one of the remaining incidence points.
In total, with careful consideration of the orderings, we therefore find
\begin{multline*}
\big\langle \mu_m , \nu_m, \rho_m \big\rangle^{S \times \p^1/\{ 0,1,\infty\}, \bullet}_{g,(m+n,\beta_h)}\\
= \frac{1}{(n! m!)^2} \sum_{k=0}^{m} \binom{m}{k} \binom{m}{k} (m-k)!
\Big( (n! k!)^2 \big\langle \mu_k, \nu_k, \rho_k \big\rangle^{S \times \p^1/ \{ 0,1,\infty \}}_{g+(m-k),(k+n,\beta_h)} \Big) \,.
\end{multline*}
The first part of Theorem~\ref{mainthm_1b} follows now by summing up.

In case $n=0$ we will use the relative condition
$\mu_m = \{ (1, x_1) , \ldots , (1,x_m) \}$ for some generic points $x_1, \ldots, x_m \in S$.
Consider an irreducibe curve $\Sigma \subset S \times \p^1$ of degree $k$ over $\p^1$
which is incident to $k$ of the points $\{ x_i \}$ over $0$.
Since $\Hilb^k(S)$ is not uniruled for every $k>0$ the map $\p^1 \to \Hilb^k(S)$ corresponding to $\Sigma$
must be constant and hence $k=1$ and $\Sigma = x_{i} \times \p^1$.
Let $f : C \to S \times \p^1$ be a possibly disconnected relative stable map
incident to $\mu_m, \nu_m, \rho_m$ over $0, 1, \infty$ respectively.
By the previous discussion the image $f(C)$ must contain the curves $x_i \times \p^1$ for all $i$
and hence meets the divisor $S_{\infty}$ in the points $x_1, \ldots, x_m$.
But if $\rho_m$ is represented by the cycle $\{ (1,F_{0}), (1,S)^{m-1} \}$ for some fiber $F_0$ disjoint from $\{ x_i \}$
this implies that $f$ is not incident to $\rho_m$ in contradiction to the assumption.
Hence the moduli space is empty and the invariant vanishes.
\qed

\subsection{Special cases in degree $2$} \label{Section_Special_cases_in_degree_2}
We will require a total of five special cases
of relative invariants of $S \times \p^1$ in degree $2$ over $\p^1$.
The first two cases are provided by Theorem~\ref{mainthm_1b} with $(m,n) = (1,1)$ and $(0,2)$.

\begin{lemma} \label{L0}
$\displaystyle
\big\langle (1,F)^2, D(F) , (1,\pt)(1, \e) \big\rangle^{S \times \p^1, \bullet}
 =
\frac{1}{2} \frac{\Theta(u,q) \cdot D \Theta(u,q)}{\Delta(q)}
$
\end{lemma}
\begin{proof}
Only maps from connected curves contribute to the invariants here, hence it is enough to consider connected invariants.
By a degeneration argument we have
\[
\big\langle (1,F)^2, D(F), (1,\pt)(1, \e) \big\rangle^{S \times \p^1}
=
\big\langle (1,F)^2 \big| \tau_0(\pt \boxtimes \omega) \tau_0(F \boxtimes \omega) \big\rangle^{S \times \p^1},
\]
which by the localization formula and the divisor axiom is
\[ \frac{u^4}{2} \Big\langle \, \BE^{\vee}(1) \, \tau_0(\pt) \, \frac{F}{1-\psi_2} \, \frac{F}{1-\psi_3} \, \Big\rangle^S \]
Degeneration of $S$ to the normal cone of an elliptic fiber $E$ yields
\[ \frac{u^4}{2} \big\langle \BE^{\vee}(1) \big| 1 \big\rangle^S
\Big( \blangle \, \omega \, \big| \frac{F}{1-\psi} \, \big| \, 1 \, \brangle^{\p^1 \times E} \Big)^2
\blangle \omega \big| \tau_0(\pt) \BE^{\vee}(1) \big| 1 \brangle^{\p^1 \times E}
\]
Using the Katz-Klemm-Vafa formula \eqref{KKV}, \eqref{assd} and Lemma \ref{P1xE_Lemma} for the first, second
and third term respectively, the claim follows.
\end{proof}

\begin{lemma} \label{L1}
$\displaystyle
\big\langle (2,\pt), D(F),  D(F) \, \big\rangle^{S \times \p^1, \bullet}
 =
\frac{1}{4} \frac{ \frac{\partial}{\partial u} \mathbf{G}(u,q) }{\Delta(q)}
$
\end{lemma}
\begin{proof}
Let $\alpha \in H^2(S, \BQ)$ be a class satisfying 
\[
\langle \alpha, \alpha \rangle = 1
\quad \text{and} \quad
\langle \alpha, F \rangle = \langle \alpha, W \rangle = 0 \,.
\]
Then apply Proposition \ref{Proposition_take_out_div}
twice with $(\lambda_1, \ldots , \lambda_4)$ equal to
\[
\big( (2, \alpha), D(F), D(F), D(\alpha) \big)
\ \text{and} \ 
\big( (1,F)(1, \alpha), D(F), (2,1), D(\alpha) \big)
\]
respectively, and use Proposition \ref{Proposition_take_out_div} and Theorem \ref{mainthm_1b}.
\end{proof}
%
%

For the last case we will require the following Hodge integrals.
\begin{lemma} \label{Hodge_Evaluations}
\begin{align}
\Big\langle \hodge \frac{\e}{1-\psi_1} \tau_0(\pt) \Big\rangle^S
& = \frac{1}{u^2} \frac{ \mathbf{G}(u,q) - 1}{\Theta^2 \Delta} \tag{i} \\
\Big\langle \hodge \frac{\e}{1-\psi_1} \Big\rangle^S
& = \frac{-2}{u^2 \Delta(q)} \tag{ii}
\end{align}
\vspace{-17pt}
\begin{multline}
\tag{iii}
\ \ \ \ \Blangle \BE^{\vee}(1) \frac{\pt}{1-\psi_1} \frac{\e}{1-\psi_2} \tau_0(\pt) \Brangle^{S} \\
=
\Big\langle \hodge \frac{\pt}{1 - \psi_1} \cdot \frac{\e}{1-\psi_2} \Big\rangle^{S} \cdot \frac{D \Theta}{\Theta}
+
\frac{( \mathbf{G}-1 )^2}{u^4 \Theta^2 \Delta}
+ 2 \frac{(\mathbf{G}-1)}{u^4 \Delta} \cdot \frac{D \Theta}{\Theta} .
\end{multline}

\end{lemma}
\begin{proof}
\noindent {\bf (i)} Consider the connected invariant 
\begin{equation} \blangle \tau_0(F \boxtimes \omega) \tau_0(\pt \boxtimes \omega) \brangle^{S \times \p^1}_{g, (\beta_h, 1)} \,. \label{402b} \end{equation}
Applying the localization formula to \eqref{402b} with $\tau_0(F \boxtimes \omega)$ specializing to
the fiber over $0$, and $\tau_0(\pt \boxtimes \omega)$ specializing to the fiber over $\infty$,
yields 
$\blangle \hodge \frac{\pt}{1-\psi_1} \brangle^S_{g, \beta_h} + \blangle \hodge \frac{F}{1-\psi_1} \tau_0(\pt) \brangle^{S}_{g, \beta_h}$.
Specializing both insertions to the fiber over $\infty \in \p^1$ yields
$\blangle \hodge \frac{\e}{1-\psi_1} \tau_0(F) \tau_0(\pt) \brangle^{S}_{g, \beta_h}$.
Since the result in both computations is the same, the claim now follows by the divisor
equation and Theorem~\ref{mainthm_1}.

\Extended{
We consider the case where both insertions specialize to $\infty$.
Let $t$ be the weight of the tangent space at $\infty$, and $-t$ at $0 \in \p^1$.
The component of a fixed map that maps non-zero in the K3 direction maps into the fiber over $\infty$.
There are two cases: Either there is no component mapping to $0$,
or there is a contracted genus $1$ component mapping to $0$.
In the second case, the euler class of the normal bundle is
\[ e(N^{\text{vir}}) = \frac{(t - \psi) (-t - \psi^{(0)}_2)}{1} \frac{ t (-t) t (-t) }{t (-t) t^g \BE^{\vee}(t^{-1}) (-t - \lambda^{(0)}_1)} \]
where we used a supscript for all classes which arise from the
factor of the moduli space which parametrize the components over $0$.
Hence, the contribution is
\[ \sum_{\eta} \Blangle \frac{\eta}{t - \psi_1} \frac{-1}{t^2} \tau_0(F) \tau_0(\pt) t^2 t^g \BE^{\vee}(t^{-1}) \Brangle^S
 \cdot \Blangle \frac{\eta^{\vee}}{-t - \psi_2} (-t - \lambda_1) \Brangle^S_{g=1, \beta=0} \,.
\]
Since $[ \Mbar_{1,1}(S,0) ]^{\text{vir}} = c_2(S) \cap [\Mbar_{1,1} \times S]$
we must have $\eta^{\vee} = 1$ and hence
\begin{align*}
\Blangle \frac{\eta^{\vee}}{-t - \psi_2} (-t - \lambda_1) \Brangle^S_{g=1, \beta=0}
& =  \Blangle \frac{1 + \lambda_1/t}{1 + \psi_1/t} \Brangle^S_{g=1, \beta=0} = 0,
\end{align*}
since $\lambda_1 = \psi_1$ on $\Mbar_{1,1}$.
Hence the localization graph with the contracted genus $1$ component does not contribute.
The first case is straightforward.
}

\noindent {\bf (ii)} Applying the localization formula to
$\blangle \tau_0(F \boxtimes \omega)^3 \brangle^{S \times \p^1}_{g, (\beta_h,1)}$
where all insertions specialize to the fiber over $0$ yields
$\blangle \hodge \frac{\e}{1 - \psi_1} \tau_0(F)^3 \brangle^S_{g, \beta_h}$.
The claim now follows from Proposition~\ref{vanishing1}, Theorem~\ref{mainthm_1} and the divisor axiom.

\noindent {\bf (iii)} 
Consider the degeneration
\begin{equation} \label{degeneration_22}
S \leadsto S \cup (\p^1 \times E) \cup (\p^1 \times E) \,.
\end{equation}
We apply the degeneration formula to the invariants
$\blangle \hodge \frac{\e}{1-\psi_1} \tau_0(\pt) \brangle^S$
where we specialize $\tau_0(\pt)$ to the first copy of $\p^1 \times E$.
Using Lemma \ref{P1xE_Lemma} the result is
\begin{equation}
\begin{aligned} \label{rtergerg}
& \Big\langle \hodge \frac{\e}{1-\psi_1} \tau_0(\pt) \Big\rangle^S \\
= \ \ &
\Blangle \hodge \frac{\e}{1-\psi_1} \Big| 1 \Brangle^S
\Blangle \omega \Big| \hodge \tau_0(\pt) \Big| 1 \Brangle^{\p^1 \times E} \\
+ \ &
\Blangle \hodge \Big| 1 \Brangle^S
\Blangle \omega \Big| \hodge \frac{\e}{1 - \psi_1} \tau_0(\pt) \Big| 1 \Brangle^{\p^1 \times E} \\
+ \ & 
\Blangle \hodge \Big| 1 \Brangle^S
\Blangle \omega \Big| \hodge \tau_0(\pt) \Big| 1 \Brangle^{\p^1 \times E}
\Blangle \omega \Big| \hodge \frac{\e}{1-\psi_1} \Brangle^{\p^1 \times E} \,.
\end{aligned}
\end{equation}

By (ii) and using the degeneration $S \leadsto S \cup (\p^1 \times E)$ we have
\begin{multline*}
\frac{-2}{u^2 \Delta(q)}
= \Big\langle \hodge \frac{\e}{1-\psi_1} \Big\rangle^S \\
=
\Blangle \hodge \frac{\e}{1-\psi_1} \Big| 1 \Brangle^S
+
\Blangle \hodge \Big| 1 \Brangle^S
\Blangle \omega \Big| \hodge \frac{\e}{1-\psi_1} \Brangle^{\p^1 \times E}
\end{multline*}
Inserting this into \eqref{rtergerg}, using (i), the Katz-Klemm-Vafa formula \eqref{KKV}, and Lemma \ref{P1xE_Lemma},
we obtain
\begin{equation} \label{400}
\Blangle \omega \Big| \hodge \frac{\e}{1 - \psi_1} \tau_0(\pt) \Big| 1 \Brangle^{\p^1 \times E}
= \frac{1}{u^2} \big( \mathbf{G} - 1 + 2 \Theta \cdot D \Theta \big).
\end{equation}

We apply the degeneration formula for \eqref{degeneration_22}
to $\blangle \BE^{\vee}(1) \frac{\pt}{1-\psi_1} \frac{\e}{1-\psi_2} \tau_0(\pt) \brangle^S$.
We specialize the marked point carrying the $\tau_0(\pt)$ insertion to the first copy of $\p^1 \times E$,
and the marked point with insertion $\pt / (1-\psi_1)$ to $S$.
The result is
\begin{align*}
& \Blangle \BE^{\vee}(1) \frac{\pt}{1-\psi_1} \frac{\e}{1-\psi_2} \tau_0(\pt) \Brangle^{S} \\
=\ \  & 
\Blangle \hodge \frac{\pt}{1-\psi_1} \frac{\e}{1-\psi_2} \Big| 1 \Brangle^S
\Blangle \omega \Big| \hodge \tau_0(\pt) \Big| 1 \Brangle^{\p^1 \times E} \\
+ \  & 
\Blangle \hodge \frac{\pt}{1-\psi_1} \Big| 1 \Brangle^S
\Blangle \omega \Big| \hodge \frac{\e}{1 - \psi_1} \tau_0(\pt) \Big| 1 \Brangle^{\p^1 \times E} \\
+ \ & 
\Blangle \hodge \frac{\pt}{1-\psi_1} \Big| 1 \Brangle^S
\Blangle \omega \Big| \hodge \tau_0(\pt) \Big| 1 \Brangle^{\p^1 \times E}
\Blangle \omega \Big| \hodge \frac{\e}{1-\psi_1} \Brangle^{\p^1 \times E} .
\end{align*}
which by a similar argument as before, and with \eqref{400} and Lemma \ref{P1xE_Lemma} is
\[
\Big\langle \hodge \frac{\pt}{1 - \psi_1} \cdot \frac{\e}{1-\psi_2} \Big\rangle^{S} \cdot \frac{D \Theta}{\Theta}
+
\frac{( \mathbf{G}-1 )^2}{u^4 \Theta^2 \Delta}
+ 2 \frac{(\mathbf{G}-1)}{u^4 \Delta} \cdot \frac{D \Theta}{\Theta} . \qedhere
\]
\end{proof}

\begin{lemma} \label{GDFDG} The series
\begin{multline*}
-4 u^4 \Blangle \hodge \frac{\pt}{1-2 \psi_1} \Brangle^S
+ \frac{1}{2} u^6 \Blangle \hodge \frac{\pt}{1-\psi_1} \frac{\pt}{1 - \psi_2} \Brangle^S \\
+ u^4 \Blangle \hodge \frac{\pt}{1 - \psi_1} \frac{\e}{1 - \psi_2} \Brangle^S
+ u^4 \Blangle \hodge \frac{\pt}{1 - \psi_1} \Brangle^S .
\end{multline*}
is equal to $\big(-2 (\mathbf{G}-1) + \Theta \cdot D \Theta \big)\frac{1}{\Delta}$.
\end{lemma}
\begin{proof}
Consider the connected invariant
\begin{equation} \label{401} \Blangle \tau_0(\pt \boxtimes \omega) \tau_0(F \boxtimes \omega)^3 \Brangle^{S \times \p^1}. \end{equation}

We apply the localization formula to \eqref{401},
with exactly two of the four insertions specializing to the fiber over $0 \in \p^1$.
The result is
\[ u^4 \Blangle \hodge \frac{\pt}{1-\psi_1} \frac{F}{1-\psi_2} \Brangle^S
 + u^4 \Blangle \hodge \frac{F}{1-\psi_1} \frac{F}{1-\psi_2} \tau_0(\pt) \Brangle^S \,,
\]
which, by a degeneration argument, Theorem \ref{mainthm_1} and Lemma \ref{P1xE_Lemma}, is equal to
$\big( (\mathbf{G}-1) + \Theta \cdot D \Theta \big)/\Delta$.

We apply the localization formula a second time to \eqref{401},
this time specializing the insertion $\tau_0(\pt \boxtimes \omega)$ to the fiber over $\infty$,
and all insertions $\tau_0(F \boxtimes \omega)$ to the fiber over $0$.
The result is
\begin{multline*}
-4 u^4 \Blangle \hodge \frac{\pt}{1-2 \psi_1} \Brangle^S
+ u^4 \Blangle \hodge \frac{\pt}{1 - \psi_1} \frac{\e}{1 - \psi_2} \tau_0(F)^3 \Brangle^S \\
+ \frac{1}{2} u^6 \Blangle \hodge \frac{\pt}{1-\psi_1} \frac{\pt}{1 - \psi_2} \Brangle^S 
+ u^4 \Blangle \hodge \frac{\pt}{1 - \psi_1} \Brangle^S .
\end{multline*}
The claim follows now follows by applying the divisor axiom to the second term
and using Theorem~\ref{mainthm_1}.
\end{proof}

\Extended{
We prove the following evaluation:

\noindent \textbf{Lemma.}
\begin{equation} 
\Big\langle \BE^{\vee}(1) \frac{\pt}{1 - 2 \psi_1} \Big\rangle^S
= \frac{1}{8} \frac{ \frac{\partial}{\partial u} \mathbf{G}(u,q) }{ u^3 \Delta(q)}
+ \frac{1}{4} \frac{(\mathbf{G}^2-1)}{u^2 \Theta^2 \Delta} \tag{a}
\end{equation}
\vspace{-16pt}
\begin{multline}
\ \Big\langle \hodge \frac{\pt}{1 - \psi_1} \cdot \frac{\e}{1-\psi_2} \Big\rangle^{S} \\
=
\left( -2 (\mathbf{G}-1) + \Theta \cdot D \Theta + \frac{1}{2} u \frac{d}{du}(\mathbf{G}) + \frac{1}{2} u^2 \frac{\mathbf{G}^2-1}{\Theta^2} \right) \frac{1}{u^4 \Delta(q)}
\tag{b}
\end{multline}
\begin{equation} \tag{c}
\begin{aligned}
\Big\langle \BE^{\vee}(1) \frac{\pt}{1 - 2 \psi_1} \Big\rangle^S
& = \frac{-1}{4 u^2 \cdot (iu)} \frac{\varphi_{2,0}}{\Delta(q)} + \frac{1}{2u^2 \Delta(q)} + \frac{(\mathbf{G}-1)}{2 u^2 \Delta(q)} \\
& = \frac{ \frac{d}{du}\big( \Theta(u,q)^2 \big) }{4 u^3 \Delta(q)} + \frac{\mathbf{G}(u,q)}{2u^2 \Delta(q)}.
\end{aligned}
\end{equation}

\begin{proof}

\noindent {\bf (a)}
By a degeneration argument and since only connected curves contribute, we have
\[
\Big\langle\, (2,\pt), \, D(F),  \, D(F) \, \Big\rangle^{S \times \p^1, \bullet}
=
\Big\langle\, (2,\pt)\, \Big| \, \tau_0(F \boxtimes \omega)^2 \, \Big\rangle^{S \times \p^1} .
\]
The left hand side has been determined in Lemma \ref{L1}.
We apply the localization formula to the right hand side. Three different terms contribute.

Let $t$ be the weight of the tangent space at $\infty$, and $-t$ at $0 \in \p^1$.
We specialize the insertions $\tau_0(F \boxtimes \omega)$ to the point over $\infty$.

\noindent \textbf{Term 1.} The first term arises from the fixlocus
parametrizing maps which consist of a degree $2$ Galois cover of class $(2,0)$
which meets the divisor $0 \times S$ with degree $2$ at a fixed point $P$.
The remaining curve maps to the fiber over $\infty$.

The Euler class of the virtual normal class is
\[ e( N^{\text{vir}} ) = \frac{ \frac{t}{2} - \psi_1 }{1} \cdot \frac{ \frac{t^2}{2} \cdot t }{t \cdot (t^g + \ldots + (-1)^g \lambda_g )} \]
where the terms arise according to the order
\[ e(N^{\text{vir}} ) =
\frac{e(\text{Def}(C)}{e( \text{Aut}(C) )} \cdot
\frac{ \prod_i e( H^0(C_i, f^{\ast} T_{S \times \p^1}) ) }{\prod_x e( H^0(x, f^{\ast}(T_{S \times \p^1})))
 \cdot e(H^1(C, f^{\ast}(T_{S \times \p^1}) ))}
\]
where in each term it is understood to only take the \emph{moving} part of the corresponding space.
Below, we may also write
\[ t^g \BE^{\vee}(t^{-1}) = t^g - \lambda_1 t^{g-1} + \ldots + (-1)^g \lambda_g \,. \]
The insertions restrict to the fixpoint locus as
\[ \tau_0([ \textbf{0} ] \boxtimes F)|_{\Mbar^{\text{fix}}} = \tau_0(F) t \,. \]
In total the contribution of this component is therefore
\begin{equation} \label{ABC}
\frac{1}{2} \left\langle\  \frac{\ev_1^{\ast}(\pt)}{\frac{t}{2} - \psi_1}
\cdot \frac{2}{t^2} \cdot t^g \BE^{\vee}(t^{-1}) \cdot \tau_0(F)^2 t^2 \ \right\rangle^S_g \,.
\end{equation}
Here the factor $\frac{1}{2}$ arises since the degree $2$ Galois cover
has a $\BZ_2$ automorphism which we have to factor out when restricting to the fix locus (which has no automorphism);
compare also \cite[page 506]{GP} where we take into account of the automorphisms
by dividing out by $|\Aut(\Gamma)|$.

A second observation is a degree check. The total degree of the insertion in \eqref{ABC} is
\begin{multline*}
\deg{\tau_0(\pt)} - \deg(t - \psi_1) - \deg(t^2) + \deg(t^g) - \deg(\BE^{\vee}(t^{-1})) + 2 \deg(\tau_0(F)) \\ + \deg(t^2) \ =\  
2 - 1 - 2 + g - 0 + 2 + 2 \ =\  g + 3
\end{multline*}
which equals the virtual dimension of $\Mbar_{g,3}(S, \beta)$ as expected.
Taking the limit $t \mapsto 0$ and applying twice the divisor equation for $\tau_0(F)$ the term \eqref{ABC} becomes
$2 \langle \BE^{\vee}(1) \frac{\pt}{1 - 2 \psi_1} \rangle^S_g$.

\noindent \textbf{Term 2.}
The second component we consider parametrizes
maps which split into three different parts:
\begin{itemize}
 \item[(i)] A map to the rubber geometry $R \times S$ relative to $0, \infty$, where $R = \p^1 / \BC^{\ast}$ denotes the rubber.
The ramification condition over $0$ is $(2, \pt)$, and over $\infty$ it is $(1,1)^2$.
The class is $(2,0)$. There is a single map in the moduli space,
consisting of a degree $2$ cover of $\p^1$ with ramification points $0$ and a point $\p^1 \setminus \{ 0, \infty \}$
which is fixed by the $\BC^{\ast}$-action of the rubber.
The map admits a $\BZ_2$ automorphism which arises by interchanges the two branches over $\p^1$.
\item[(ii)] Two connected components, each mapping with degree $1$ to
a fixed rational line $\p^1 \times \pt$ inside $S \times \p^1$,
and both glued to the curve in (i) along the ralative points over $\infty$.
These are the only components which maps into the 'interior' of $S \times \p^1$.
\item[(iii)] A component mapping entirely into the K3 surface $S$ with class $\beta_h$,
glued to one of the components in (ii).
\end{itemize}
The Euler class of the virtual normal bundle is
\[ e(N^{\text{vir}}) = \frac{ (-t - \Psi_{\infty}) (t- \psi_1) }{t} \cdot \frac{t^2 \cdot t}{t \cdot t^g \BE^{\vee}(t^{-1})} \,. \]
Here the class $\Psi_{\infty}$ describes the first Chern class of the
relative divisor at $\infty$ of the rubber bubble, see \cite{GV, MP}.
Note, that although we have two relative point and so have two points along which we glue the curve,
the class $(-t-\Psi_{\infty})$ only appears once; this is related to the
predeformability conditions of relative stable maps: the only way one of the node
is smoothed is by smoothing the nodal divisor of $(R \times S) \cup (S \times \p^1)$,
in which case both curves get smoothed.

The total contribution is therefore
\[
2 \cdot \Big\langle (2, \pt) \Big| \frac{1}{-t - \Psi_{\infty}} \Big| (1,e)^2 \Big\rangle^{R \times S}_{g=0}
\cdot
\Big\langle \frac{\ev_1^{\ast}(\pt)}{t - \psi_1} \frac{t^g \BE^{\vee}(t^{-1})}{t} \tau_0(F)^2 t^2 \Big\rangle^S_g \,.
\]
Here the factor $2$ arises by considering automorphisms as follows:
Recall first that we follow the convention of unordered relative points,
hence the relative insertion $(1,\e)^2$ over $\infty \in R$
is unordered.
Similarly, the attachment points over $0 \in \p^1$ are considered to be unordered;
and it does not matter to which rational line we attach part (iii) of the curve.
In the glueing over $\infty \in R$ both sides are unordered;
hence we have to multiply by a factor of $\Fz( (1,\e)^2 ) = 2$
to compensate for dividing out the symmetries twice.
(Breaking the curve into two pieces yields a set of unordered \emph{pairs}
of points on each side; hence after taking the pairs apart we have either two sets of unordered points times the order
of the symmetry group, or two sets of ordered points divided out by the symmetry group.)
As we will see in a second the rubber integral contributes $1/2$ leaving us the whole factor $1$.
(If we would have ordered the relative points, the
rubber integral would have no automorphism, hence contributing $1$.
For the relative conditions the $i$-marked points over $\infty \in R$ is canonically identified with
the $i$-marked point over $0 \in \p^1$.
And there are two choices for attaching the map to the K3 surface to one of the two rational lines.
Finally, we need to componsate for picking an ordering at the glueing points;
hence we need to divide by the automorphisms of the relative conditions which is $2$.
In total we obtain multiplicity $1$.)

Finally, there are two choices for attaching the map to the K3 surface. In total therefore again multiplicity $1$.)

The contribution of the rubber term ($1/2$) can be seen
directly since the moduli space is a point with automorphism $2$.
Alternatively, we may use the evaluation (see \cite[Theorem 2]{OP1})
\[ \blangle \mu \big| \tau_{l(\mu)+l(\nu)-2}(\omega) \big| \nu \rangle^{\p^1}_{g=0} = \frac{1}{ |\Aut(\mu)| \cdot |\Aut(\nu)|} \]
where the degree is specified by $\mu$ and $\nu$. In our case $\mu = (2)$ and 
$\nu = (1,1)$ which gives $1/2$.

We find Term 2 yields $-\langle \BE^{\vee}(1) \frac{\pt}{1 - \psi_1} \rangle^S_g$.

\vspace{5pt}
\noindent \textbf{Term 3.}
As in Term $2$ above, the curve has three components. Components (1) and (2) are as above, while
(3) is glued to both components of (2). The dual graph of the curve has therefore a loop,
and the map to the K3 surface is of genus $g-1$.
Also, we obtain an additional automorphism for the part of the curve mapping to the interior of $S \times \p^1$,
since we may now the two components in (ii).

The Euler class of the normal bundle is
\[ e(N^{\text{vir}}) = \frac{(-t-\Psi_{\infty}) (t-\psi_1) (t-\psi_2)}{1} \cdot \frac{t^2 \cdot t}{t^2 \cdot t^g \BE^{\vee}(t^{-1}) } \]
which yields the contribution
$- \frac{1}{2} \langle \frac{\pt}{1-\psi_1} \cdot \frac{\pt}{1-\psi_2} \cdot \BE^{\vee}(1) \rangle^S_{g-1}$.

The result is
\begin{multline*}
 2 u^3 \Big\langle \BE^{\vee}(1) \frac{\pt}{1 - 2 \psi_1} \Big\rangle^S - u^3 \Big\langle \BE^{\vee}(1) \frac{\pt}{1 - \psi_1} \Big\rangle^S \\
- \frac{1}{2} u^5 \Big\langle \BE^{\vee}(1) \frac{\pt}{1 - \psi_1} \frac{\pt}{1 - \psi_2} \Big\rangle^S \,.
\end{multline*}
The claim now follows from Theorem~\ref{mainthm_1}.

\noindent \textbf{(b)} This follows from Lemma \ref{GDFDG} after solving using (a).

\noindent \textbf{(c)}
By degeneration and the GW/Hilb correspondence
\begin{align*}
\blangle (2,F) \big| \tau_0(\omega F)^3 \brangle^{S \times \p^1}
& = 2 \blangle (2,F) \big| (1,F)^2 \brangle^{S \times \p^1, \sim} \\
& = \frac{1}{2 I} \blangle \Fp_{-2}(F) \big| \Fp_{-1}(F)^2 \brangle^{\Hilb} \\
& = - \frac{1}{2 i} \frac{1}{\Delta(q)} \varphi_{2,0}
\end{align*}
By localization
\begin{align*}
& \blangle (2,F) \big| \tau_0(\omega F)^3 \brangle^{S \times \p^1} \\
& = 2 \Blangle \frac{F}{1-2 \psi_1} \hodge \Brangle^{S} u^3 - \Blangle \frac{F}{1-\psi_1} \hodge \Brangle^S u^3
- \Blangle \frac{F}{1-\psi_1} \frac{\pt}{1-\psi_2} \hodge \Brangle^{S} u^5 \\
& = 2 \Blangle \frac{F}{1-2 \psi_1} \hodge \Brangle^{S} u^3
- \frac{u^1}{\Delta(q)} - \frac{u (\mathbf{G}-1)}{\Delta(q)}.
\end{align*}
The claim now follows from
\[ \varphi_{2,0} = 2 \Theta^2 \left( \frac{1}{i} \frac{d}{du} \log(\Theta) \right) . \]
\end{proof}
}

We determine the fifth special case.
\begin{lemma} \label{L2}
$\displaystyle
\big\langle D(\pt), D(F), D(\pt) \big\rangle^{S \times \p^1, \bullet}
 =
 \frac{ \left( D \Theta(u,q) \right)^2 }{\Delta(q)}
$
\end{lemma}
\begin{proof}
Only connected curves contribute to the integral. The degeneration formula yields
\begin{multline*}
\blangle \tau_0(\pt \boxtimes \omega)^2 \tau_0(F \boxtimes \omega) \brangle^{S \times \p^1}
=
\big\langle\, D(\pt), \, D(F), \, D(\pt) \, \big\rangle^{S \times \p^1} \\
+ 2 \blangle (1, \pt)^2 \big| \tau_0(F \boxtimes \omega) \brangle^{S \times \p^1}
+ 2 \blangle (1, \pt)(1, F) \big| \tau_0(F \boxtimes \omega) \brangle^{S \times \p^1} \,.
\end{multline*}
The last two terms of the right hand side are computed directly using
the localization formula and Theorem \ref{mainthm_1}:
\begin{align*}
\blangle (1, \pt)^2 \big| \tau_0(F \boxtimes \omega) \brangle^{S \times \p^1}
& =
\frac{1}{2} \frac{(\mathbf{G}-1)^2}{\Theta^2 \Delta} \\
\blangle (1, \pt)(1, F) \big| \tau_0(\pt \boxtimes \omega) \brangle^{S \times \p^1}
& = 
\frac{(\mathbf{G}-1)}{\Delta} \frac{ D \Theta}{\Theta}
\end{align*}
Hence it remains to prove
\[
\blangle \tau_0(\pt \boxtimes \omega)^2 \tau_0(F \boxtimes \omega) \brangle^{S \times \p^1}
= 
\frac{ ( D \Theta )^2 }{\Delta} + 2 \frac{(\mathbf{G}-1)}{\Delta} \frac{ D \Theta}{\Theta} + \frac{(\mathbf{G}-1)^2}{\Theta^2 \Delta} \,.
\]

We apply the localization formula to the left hand side,
specializing exactly one of the $\tau_0(\pt \boxtimes \omega)$ insertions
to the fiber over $0$, and the other insertions to the fiber over $\infty$.
Five fixed loci contribute. The result is
\begin{equation} \label{GERGE}
\begin{aligned} 
- \ & 4 u^4 \Blangle \hodge \frac{\pt}{1-2 \psi_1} \tau_0(\pt) \Brangle^{S} \\
+ \ & u^4 \Blangle \BE^{\vee}(1) \frac{\pt}{1-\psi_1} \frac{\e}{1-\psi_2} \tau_0(\pt) \tau_0(F) \Brangle^{S} \\
+ \ & \frac{1}{2} u^6 \Blangle \hodge \frac{\pt}{1-\psi_1} \frac{\pt}{1-\psi_2} \tau_0(\pt) \Brangle^{S} \\
+ \ & u^4 \Blangle \hodge \frac{\pt}{1-\psi_1} \tau_0(\pt) \Brangle^S \\
+ \ & u^4 \Blangle \hodge \frac{\pt}{1-\psi_1} \frac{F}{1-\psi_2} \tau_0(\pt) \Brangle^{S} \,.
\end{aligned}
\end{equation}
By the divisor equation and Lemma~\ref{Hodge_Evaluations}(iii),
we may remove the $\tau_0(F)$ and $\tau_0(\pt)$ insertion from the second term.
Since only fiber and point classes appear in the other terms of \eqref{GERGE},
the $\tau_0(\pt)$ insertion can be degenerated off to a copy of $\p^1 \times E$, where it is evaluated by Lemma~\ref{P1xE_Lemma}.
The remaining first four terms then exactly yield the evaluation of Lemma \ref{GDFDG}.
Applying Theorem~\ref{mainthm_1} for the last term, a direct calculation shows the claim.
\end{proof}

\Extended{
The localization calculation in the previous Lemma proceeds as follows.
The five terms are the following:

\noindent \textbf{(1)}. The localization graph is
\[
\begin{tikzpicture}
\vertex (a) at (0,0) [label=left:$\pt$] {};
\node[shape=circle,draw, inner sep=2pt] (b) at (3,0) {$\pt, F$};
\draw (a) edge node[above]{$2$} (b);
\end{tikzpicture}
\]
where the right hand side denotes the fiber over $\infty$ with
tangent space carrying weight $t$, and the left side denotes the fiber over $0$.
The round non-filled circle denotes the vertex which carries
the class $\beta_h$, i.e. which parametrizes non-constant maps to the K3 surface $S$.
If not otherwise denotes, the genus labeling of the round circle vertex is $g$,
and of all other vertices is $0$.
The edge is labeled $2$ since it correspond to a degree $2$ Galois cover
(i.e. a map $\p^1 \to \p^1$ of the form $z \mapsto z^2$);
an unlabeled edge will correspond to a degree $1$ map.

The Euler class of the moving part of the normal bundle is
\[ e(N^{\text{vir}}) = \frac{\left(\frac{t}{2} - \psi_1 \right)}{1} \frac{ t^4/4 \cdot t}{t t^g \BE^{\vee}(t^{-1})} . \]
The automorphism factor arising from the degree $2$ Galois cover is $2$.
Hence, the contribution of this term is
\begin{multline*}
\frac{1}{2} u^4 \Blangle \frac{(-t) \Fp}{\frac{t}{2} - \psi_1} \frac{4}{t} t^{g} \BE^{\vee}(t^{-1}) \tau_0(\pt) \tau_0(F) t^2 \Brangle^S \\
= -4 u^4 \Blangle \frac{\pt}{1-2 \psi_1} \hodge \tau_0(\pt) \tau_0(F) \Brangle^S
\end{multline*}
The factor $u^4$ arises since we weight the genus $g$ invariants on $S \times \p^1$
by $u^{2g+2}$, while on the K3 surface $S$ we weight them with $u^{2g-2}$,
and after localization the genus has not changed.

\noindent \textbf{(2)}.

\[
\begin{tikzpicture}
\vertex (a) at (0,1) [label=left:$\pt$] {};
\node[shape=circle,draw, inner sep=2pt] (b) at (3,1) [label=below:$\tiny{g-1}$] {$\pt, F$};
\draw (a) edge[bend right=20] (b);
\draw (a) edge[bend left=20] (b);
\end{tikzpicture}
\]

Here the left vertex is $\Mbar_{0,3}$ and is therefore stable.
The two incoming edges will meet at distinct marked points.
The graph has a loop, hence the genus on the right vertex is $g-1$.
The Euler class is
\begin{multline*} e(N^{\text{vir}}) \\
= 
\frac{(t-\psi_1) (t - \psi_2) (-t-\psi_1^l)(-t-\psi_2^l)}{1}
 \frac{(t (-t))^2 t (-t)}{t^2 (-t)^2 (t^{g-1} + \ldots + (-1)^{g-1} \lambda_{g-1})},
\end{multline*}
where in the numerator of the second factor the term $(t (-t))^2$ arises from the two edges,
$t$ arises from the tangent space to the right vertex, $-t$ arises from the tangent space over the left vertex;
in the denominator $t^2$ arises from the two marked points
on the right hand side, $(-t)^2$ arises from the two marked points
on the left vertex.
On $\Mbar_{0,3}$ the psi classes vanish, hence $\psi_1^l = \psi_2^l = 0$.

The automorphism factor is $2$ coming from interchanging the two edges.
Hence, the contribution is
\begin{multline*}
\frac{1}{2} u^6 \blangle (-t) \tau_0(\pt) \tau_0(1)^2 \frac{1}{(-t)^2} \brangle^{S}_{g=0, \beta=0} \cdot \\
\Blangle \frac{\pt}{t-\psi_1} \frac{\pt}{t-\psi_2}
\left( \frac{-1}{t^2} \right) (t^g + \ldots + (-1)^g \lambda_g) \tau_0(F) \tau_0(\pt) t^2 \Brangle^S 
\end{multline*}
where the factor $u^6$ corrects for the lowered genus.
Note again that all the degrees of the insertion matches.
The above simplifies to the final contribution:
\[ \frac{1}{2} \Blangle \hodge \frac{\pt}{1-\psi_1} \frac{\pt}{1-\psi_2} \tau_0(F) \tau_0(\pt) \Brangle^S \,. \]

\noindent \textbf{(3)}.

\[
\begin{tikzpicture}
\vertex (a) at (0,0.4) [label=left:$\pt$] {};
\vertex (a2) at (0,-0.4) {};
\node[shape=circle,draw, inner sep=2pt] (b) at (3,0) {$\pt, F$};
\draw (a) edge (b);
\draw (a2) edge (b);
\end{tikzpicture}
\]

Here (although it looks different) both edges represent maps
to horizontal rational curves $s \times \p^1$ for some $s \in S$.
The Euler class is
\[ e(N^{\text{vir}}) = \frac{(t-\psi_1) (t-\psi_2)}{(-t)} \frac{(t (-t))^2 t}{t^2 (t^g + \ldots + (-1)^g \lambda_g)} . \]
The automorphism factor is $1$ (since one of the edges carries a marked point,
and the other doesnt).
Hence, we find
\begin{multline*}
u^4 \Blangle \frac{(-t) \pt}{t-\psi_1} \frac{\e}{t-\psi_2} \left( \frac{-1}{t^2} \right) (t^g + \ldots + (-1)^g \lambda_g)
\tau_0(F) \tau_0(\pt) t^2 \Brangle^S \\
= u^4 \Blangle \hodge \frac{\pt}{1-\psi_1} \frac{\e}{1-\psi_2} \tau_0(F) \tau_0(\pt) \Brangle^S.
\end{multline*}

\noindent \textbf{(4)}.

\[
\begin{tikzpicture}
\vertex (a) at (0,0) [label=left:$\pt$] {};
\node[shape=circle,draw, inner sep=2pt] (b) at (3,0.4) {$\pt, F$};
\vertex (b2) at (3,-0.4) {};
\draw (a) edge (b);
\draw (a) edge (b2);
\end{tikzpicture}
\]

Note again that the left vertex is a stable $\Mbar_{0,3}(S, 0) = \Mbar_{0,3} \times S$
and we have to include the smooth terms $(-t - \psi_i^l)$.
We use $\psi_i^l = 0$ from the start.
The Euler class then is
\[ e(N^{\text{vir}}) = \frac{(t-\psi_1) (t-\psi)(-t)^2}{t} \frac{(t(-t))^2 t (-t)}{t (-t)^2 (t^g + \ldots + (-1)^g \lambda_g)} \]
and the contribution is
\[ u^4 \Blangle \hodge \frac{\pt}{1-\psi_1} \tau_0(\pt) \Brangle^S . \]

\noindent \textbf{(5)}.

\[
\begin{tikzpicture}
\node[shape=circle,draw, inner sep=3pt] (a) at (0,0) {$\pt$};
\vertex (b1) at (3,0.4) [label=right:$\pt$] {};
\vertex (b2) at (3,-0.4) [label=right:$F$] {};
\draw (a) edge (b1);
\draw (a) edge (b2);
\end{tikzpicture}
\]

The Euler class is
\[ e(N^{\text{vir}}) = \frac{(-t-\psi_1) (-t-\psi_2)}{1} \frac{ (t (-t))^2 (-t)}{(-t)^2 ((-t)^g + \ldots + (-1)^g \lambda_g)} \]
which yields the contribution
\[ u^4 \Blangle \frac{\pt}{-t-\psi_1} \frac{F}{-t-\psi_2} \big( (-t)^g + \ldots + (-1)^g \lambda_g \big) \tau_0(\pt) \Brangle. \]
We specialize $t$ to $-1$.
(After expanding the rational functions above in $t$ only the coefficient of $t^0$
is non-zero and contributes. Hence, it doesnt matter to which value we specialize $t$.)
The result is
\[ u^4 \Blangle \hodge \frac{\pt}{1-\psi_1} \frac{F}{1-\psi_2} \tau_0(\pt) \Brangle. \]

This yields all the five terms described above. \\

\noindent \textbf{Calculation 2.} We compare our calculation above
with a corresponding calculation in the relative context.
Namely, we use the localization formula to compute the connected invariants
\[ \Blangle (1,\pt)(1,\e) \Big| \tau_(F \boxtimes \omega) \tau_0(\pt \boxtimes \omega) \Brangle^{S \times \p^1} . \]

The contributing terms are as follows:

\noindent \textbf{(1)}.

\[
\begin{tikzpicture}
\vertex (a) at (0,0.4) [label=left:$\pt$] {};
\vertex (a2) at (0,-0.4) [label=left: $\e$] {};
\node[shape=circle,draw, inner sep=2pt] (b) at (3,0) {$\pt, F$};
\node (c) at (0,-1) [label=below:$\mathbf{0}$] {};
\node (c) at (3,-1) [label=below:$\mathbf{\infty}$] {};
\draw (a) edge (b);
\draw (a2) edge (b);
\path[draw, snake it] (0,1) -- (0,-1);
\path[draw] (0,-1) -- (3,-1);
\end{tikzpicture}
\]

Here the left hand curly line denotes the relative divisor over $0$.
The circled node carries the class $\beta$ and lies over $\infty$.
We have
\[ e(N^{\text{vir}})
 = \frac{(t-\psi_1) (t-\psi_2)}{1} \frac{t^3}{t^2 (t^g + \ldots + (-1)^g \lambda_g)} \,.
\]
Note: There is no contribution from automorphism over $0$, since these are relative marked points.
Because we work in a relative geometry, the tangent bundle in the edge term
$H^0(C_i, f^{\ast}(T_{\p^1}))$
has to be replaced with the logarithmic tangent bundle $T_{\p^1}(- \textbf{0})$,
which yields the moving part contribution
\[
e^{\text{mov}}( H^0(C_i, f^{\ast}(T_{\p^1}(-\textbf{0})))) = t \cdot \frac{t}{2} \dots \frac{t}{d} = \frac{t^d}{d!}
\]
for a degree $d$ Galois cover.
(In the non-relative case, we would have $t (t/2) (t/3) \ldots (t/d) (-t/d) \ldots (-t)$).
Above, each connecting edge contributes $t$, and the term from the vertex over $\infty$ contributes another $t$,
hence the numerator in the second factor is $t^3$.

The contribution is therefore
\begin{multline*}
u^4 \Blangle \frac{\pt}{t-\psi_1} \frac{\e}{1-\psi_2} \frac{1}{t} (t^g + \ldots + (-1)^g \lambda_g) \tau_0(F) \tau_0(\pt) t^2 \Brangle^S \\
= u^4 \Blangle \hodge \frac{\pt}{1-\psi_1} \frac{\e}{1-\psi_2} \tau_0(F) \tau_0(\pt) \Brangle^S.
\end{multline*}

\noindent \textbf{(2)}.

\[
\begin{tikzpicture}
\vertex (a) at (0,0.2) [label=left:$\pt$] {};
\vertex (a2) at (0,-0.2) [label=left:$\e$] {};
\vertex (b) at (3,0) [label=above right: ${(2, \pt)}$] {}; 
\node (bd) at (3,-1) [label=below:$\mathbf{0}$] {};
\node (cd) at (6,-1) [label=below:$\mathbf{\infty}$] {};
\node[shape=circle,draw, inner sep=2pt] (c) at (6,0) {$\pt, F$};
%
%
\draw (b) edge node[below]{$2$} (c);
\path[draw, snake it] (3,1) -- (3,-1);
\path[draw, snake it] (0,1) -- (0,-1);
\path[draw] (0,-1) -- (3,-1) -- (6,-1);
\end{tikzpicture}
\]
Here $\eta = (2, \pt)$ denotes the glueing condition over $\mathbf{0}$
on the side of $S \times \p^1$ (not the bubble side).
The relative conditions with simple intersection are denoted
by $\gamma$ instead of $(1,\gamma)$.
The curve in the bubble is not graphed since it may be arbitrary and is not fixed
by a $\BC^{\ast}$-action.

The Euler class of the moving normal bundle is
\[
e(N^{\text{vir}}) =
\frac{(-t-\Psi_\infty) (t/2 - \psi_1)}{1} \frac{t^2/2 \cdot t}{t (t^g + \ldots + (-1)^g \lambda_g},
\]
where $\Psi_{\infty}$ is the Psi class
over the relative point over $\mathbf{0}$.
Note that the torus weight is $-t$ and not $-t/2$ over $0$,
since the corresponding line bundle
has a factor of $T_{\p^1, \mathbf{0}}$ without pulling back by any map;
the deformation space corresponds to deforming the whole bubble,
and not any single component of the map. The reasons
is the predeformability condition: deforming a single node over $\mathbf{0}$
is equivalent to deforming all component, see the paper by Graber-Vakil.
In particular, if we have several curves intersecting over $\mathbf{0}$
we also only get a single copy of $-t - \Psi_{\infty}$.

Because of the glueing over $\mathbf{0}$ we must multiply with
a factor of $\Fz(\eta)$ where $\eta$ is the relative glueing.
If $\eta = (\eta_1, \eta_2, \ldots \ )$, then
$\Fz(\eta) = | \Aut(\eta) | \prod_i \eta_i$.
The automorphism factor arises since we consider unordered partitions
on both sides, the factor $\eta_i$ arises as outcome of the local considerations
in the degeneration formula [Jun Li].

Here $\eta = (2, \pt)$ and $\Fz(\eta) = 2$. We have an automorphism factor
of $2$ coming from the degree $2$ Galois cover.
In total, we hence obtain the contribution
\begin{multline*}
\Blangle (1,\pt)(1,\e) \Big| \frac{1}{-t -\Psi_{\infty}} \Big| (2,1) \Brangle^{S \times \p^1, \sim} \cdot 2 \cdot \frac{1}{2} \cdot \\
\Blangle (t^g + \ldots + (-1)^g \lambda_g) \frac{2}{t^2} \frac{\pt}{\frac{t}{2} - \psi_1} \tau_0(\pt) \tau_0(F) t^2 \Brangle^S
\end{multline*}
where we note that the degree matches the virtual dimension across the total
insertions.

Since the left hand side insertion $(1, \pt)(1,\e)$ act as
if they were ordered points (carring different insertions), we have
\[
\Blangle (1,\pt)(1,\e) \Big| (2,\e) \Brangle^{S \times \p^1, \sim}
=
2 \blangle (1,1) \big| (2) \brangle^{\p^1, \sim} = 1.
\]
We find the total contribution is
\[
-4 u^4 \Blangle \hodge \frac{\pt}{1-2\psi_1} \tau_0(\pt) \tau_0(F) \Brangle^S.
\]

\noindent \textbf{(3)}. 

\[
\begin{tikzpicture}
\vertex (a) at (0,0.2) [label=left:$\pt$] {};
\vertex (a2) at (0,-0.2) [label=left:$\e$] {};
\vertex (b1) at (3,0.15) [label=above right: ${\pt}$] {};
\vertex (b2) at (3,-0.15) [label=below right: ${\pt}$] {};
\node (bd) at (3,-1) [label=below:$\mathbf{0}$] {};
\node (cd) at (6,-1) [label=below:$\mathbf{\infty}$] {};
\node[shape=circle,draw, inner sep=2pt] (c) at (6,0) [label=below:{\small $g-1$}] {$\pt, F$};
%
%
\draw (b1) edge (c);
\draw (b2) edge (c);
\path[draw, snake it] (3,1) -- (3,-1);
\path[draw, snake it] (0,1) -- (0,-1);
\path[draw] (0,-1) -- (3,-1) -- (6,-1);
\end{tikzpicture}
\]

The relative condition over $\mathbf{0}$ on the $S \times \p^1$-side is $(1,\pt)^2$.
Hence, we will require the following bubble evaluation:
\[ \Blangle (1,\pt)(1,\e) \Big| \Psi_{\infty} \Big| (1,\e)^2 \Brangle_{g,(2,0)}^{S \times \p^1, \bullet \sim}
 = \frac{1}{2} \delta_{g0}
\]
[\textbf{Proof}. We remove the insertion $\Psi_{\infty}$ following
[Maulik-Pandharipande, Topological View on Gromov-Witten theory].
Namely, we have
\begin{align*}
& \quad \blangle (1,\pt)(1,\e) \big| \Psi_{\infty} \Big| (1,\e)^2 \brangle_{g,(2,0)}^{S \times \p^1, \bullet \sim} \\
& =
\frac{1}{2g-2+4} \blangle (1,\pt)(1,\e) \big| \tau_1(1) \Psi_{\infty} \Big| (1,\e)^2 \brangle_{g,(2,0)}^{S \times \p^1, \bullet \sim} \\
& =
\sum_{\eta}
\frac{1}{2g+2} \blangle (1,\pt)(1,\e) \big| \tau_1(1) \Big| \eta \brangle^{\bullet \sim}_{g_1}
\Fz(\eta) \blangle \eta^{\vee} \Big| (1,\e)^2 \brangle^{\bullet \sim}_{g_2}.
\end{align*}
The second factor must be actually connected, since
other wise we must have $\eta = \{ (1,\e)^2 \}$
which is impossible by dimension reasons.
Then, in the second term the total degree in the $S$-drection must be $2$,
which shows that $\eta = (2, \e)$ or $\eta = (1, \alpha)(1,\beta)$ where $\deg(\alpha) + \deg(\beta) = 2$.
Since the second can't happen for dimension reaons, we must have $\eta = (2,\e)$.
The second term hence evaluates $\blangle (2, \pt) \big| (1,\e)^2 \brangle^{\bullet \sim}_{g_2} = \frac{1}{2} \delta_{g_2, 0}$,
which  (since $g = g_1 + g_2 + l(\eta) - 1$ and $g_2 =0$ and $l(\eta) = 1$) yields
\begin{align*}
& \quad \frac{1}{2g+2} \blangle (1,\pt)(1,\e) \big| \tau_1(1) \Big| (2, \e) \brangle^{\bullet \sim}_{g}
2 \cdot \frac{1}{2} \\
& = 
\frac{2g-2+3}{2g+2} \blangle (1,\pt)(1,\e) \big| (2, \e) \brangle^{\bullet, \sim}_{g} \\
& = \frac{2g-2+3}{2g+2} \delta_{g0} \\
& = \frac{1}{2} \delta_{g0}.
\end{align*}
\textbf{End Proof}]

In particular the genus of the map to $S$ over $\infty$ is $g-1$ as claimed in the picture.
The Euler class of the normal bundle is
\[
e(N^{\text{vir}}) = \frac{(t-\psi_1) (t-\psi_2)(-t- \Psi_{\infty})}{1}
\frac{t^2 \cdot t}{t^2 (t^{g-1} + \ldots + (-1)^{g-1} \lambda_{g-1})}
\]
We have $\Fz( (1,\pt)^2 ) = 2$ and an automorphism interchanging the two degree $1$ edges
(This is possible since the left relative edges are unordered!).
Hence, the contribution is
\begin{multline*}
u^6 \Blangle (1,\pt)(1,\e) \Big| \frac{1}{-t-\Psi_{\infty}} \Big| (1,\e)^2 \Brangle^{S \times \p^1, \sim} \cdot 2 \cdot \frac{1}{2} \\
\Blangle \frac{1}{t} (t^g + \ldots + (-1)^g \lambda_g) \frac{\pt}{t - \psi_1} \frac{\pt}{t-\psi_2} \tau_0(\pt) \tau_0(F) t^2 \Brangle^{S}
\end{multline*}
which simplifies to
\[ \frac{1}{2} u^6 \Blangle \hodge \frac{\pt}{1-\psi_1} \frac{\pt}{1-\psi_2} \tau_0(\pt) \Brangle^S. \]

\noindent Remark. We started out with connected invariants, hence
above one should be careful to see that the resulting curve that contributes is actually connected.
This follows, since in the bubble curve only connected curves contribute.

\noindent \textbf{(4)}. 

\[
\begin{tikzpicture}
\vertex (a) at (0,0.2) [label=left:$\pt$] {};
\vertex (a2) at (0,-0.2) [label=left:$\e$] {};
\vertex (b1) at (3,0.5) [label=above right: ${\pt}$] {};
\vertex (b2) at (3,-0.2) [label=below right: ${\pt}$] {};
\node (bd) at (3,-1) [label=below:$\mathbf{0}$] {};
\node (cd) at (6,-1) [label=below:$\mathbf{\infty}$] {};
\node[shape=circle,draw, inner sep=2pt] (c) at (6,0.5) {$\pt, F$};
\draw (b1) edge (c);
\path[draw] (3,-0.2) -- (6,-0.2);
\path[draw, snake it] (3,1) -- (3,-1);
\path[draw, snake it] (0,1) -- (0,-1);
\path[draw] (0,-1) -- (3,-1) -- (6,-1);
\end{tikzpicture}
\]

The Euler class is
\[ e(N^{\text{vir}})
 = \frac{(t-\psi_1) (-t-\Psi_{\infty})}{t} \frac{t^2 \cdot t}{t (t^g + \ldots + (-1)^g \lambda_g} .
\]
The automorphism factor is $1$, the glueing factor $2$, hence in total
similar to before
\[ u^4 \Blangle \hodge \frac{\pt}{1-\psi_1} \Brangle^S. \]

\noindent \textbf{Result.} We find
\begin{align*}
& \quad \quad \Blangle (1,\pt)(1,\e) \Big| \tau_(F \boxtimes \omega) \tau_0(\pt \boxtimes \omega) \Brangle^{S \times \p^1} \\
= \quad  &
u^4 \Blangle \hodge \frac{\pt}{1-\psi_1} \frac{\e}{1-\psi_2} \tau_0(F) \tau_0(\pt) \Brangle^S \\
- & 4 u^4 \Blangle \hodge \frac{\pt}{1-2\psi_1} \tau_0(\pt) \tau_0(F) \Brangle^S \\
+ & \frac{1}{2} u^6 \Blangle \hodge \frac{\pt}{1-\psi_1} \frac{\pt}{1-\psi_2} \tau_0(\pt) \Brangle^S \\
+ & u^4 \Blangle \hodge \frac{\pt}{1-\psi_1} \Brangle^S.
\end{align*}

This differs from the previous calculation by $u^4 \Blangle \hodge \frac{\pt}{1-\psi_1} \frac{F}{1-\psi_2} \tau_0(\pt) \Brangle^S$
which is precisely the term that arises in the degeneration formula.
}

\Extended{
\noindent \textbf{Lemma.} We have the evaluation
\[
\blangle \mu |\nu \brangle^{\p^1 , \bullet}
= \frac{\delta_{\mu, \nu}}{\Fz(\mu)} \,,
\]
where we use degree $|\mu|$ and the genus $g$ is determined by the insertions via
the dimension constraint; the combinatorial factor is
\[ \Fz(\mu) = | \Aut(\mu) | \cdot \prod_i \mu_i \,. \]

\noindent \textbf{Proof.}
All connected invariants $\langle \mu, \nu \rangle^{\p^1}$ vanish by dimension reasons
except for $\mu = (d)$ and $\nu = (d)$, and genus $0$.
In this exceptional case, the moduli space consists of a single point
parametrizing a map to $\p^1$ with full ramification over $0, \infty$.
The automorphism group is $\BZ_d$ (for example,
the map $z \mapsto z^d$ has automorphism obtained by multiplication with $d$th root of unities).
Hence,
\[ \langle (d), (d) \rangle^{\p^1}_{g=0,d} = \frac{1}{d} \,. \]
We immediately find, that the disconnected invariants vanish unless $\mu = \nu$.
If $\mu = \nu$, then the disconnected invariant with unordered relative markings
is $|\Aut(\mu)|^{-2}$ times the invariant with ordered markings on both relative sides.
Let $m_i = |\{ j | \mu_j = i \}|$ be the number of times $i$ appears in $\mu$.
Then there are $m_i!$ choices for matching up the $m_i$ relative points over $0$
and the $m_i$ relative points over $\infty$ which are labeled with $i$.
Hence, in total we have $\prod_i m_i! = | \Aut(\mu) |$ choices.
We conclude
\begin{multline*}
\blangle \mu ||\nu \brangle^{\p^1 , \bullet}
= \frac{\delta_{\mu,\nu}}{|\Aut(\mu)|^2} |\Aut(\mu)| \prod_i \left( \langle (i), (i) \rangle^{\p^1} \right)^{m_i} \\
= \frac{\delta_{\mu,\nu}}{|\Aut(\mu)|} \prod_j \frac{1}{\mu_j} = \frac{\delta_{\mu,\nu}}{\Fz(\mu)} \,. \qedhere
\end{multline*}
}

\subsection{Proof of Theorem \ref{thm_GWHilb_correspondence}}
We consider the case $d=2$. The invariants
\begin{equation}
\begin{gathered}
\label{special_case_invs}
\big\langle\, (1,F)^2, \, D(F) , \, (1,F)^2 \, \big\rangle^{S \times \p^1, \bullet} \\
\big\langle\, (1,\pt)(1,F), \, D(F) , \, D(F) \, \big\rangle^{S \times \p^1, \bullet} \\
\big\langle\, (1,F)^2, \, D(F) , \, D(\pt) \, \big\rangle^{S \times \p^1, \bullet} \\
\big\langle\, (2,\pt), \, D(F),  \, D(F) \, \big\rangle^{S \times \p^1, \bullet} \\
\big\langle\, D(\pt), \, D(F), \, D(\pt) \, \big\rangle^{S \times \p^1, \bullet}
\end{gathered}
\end{equation}
were computed in Theorem \ref{mainthm_1b} and Lemmas \ref{L0}, \ref{L1} and \ref{L2}. 
By comparision with the results of \cite{HilbK3},
the GW/Hilb correspondence (Conjecture \ref{GW/Hilb_correspondence})
holds in the case of the invariants \eqref{special_case_invs}.
Under the GW/Hilb correspondence the WDVV equations on the Hilbert scheme side
correspond to the relations of Proposition \ref{Proposition_WDVV_analog}.
Similarly, the divisor axiom on the Hilbert scheme side corresponds to Proposition \ref{Proposition_take_out_div} above.
A direct check shows that all degree $2$ relative invariants
\[ \big\langle \lambda_1, \lambda_2, \lambda_3 \brangle^{S \times \p^1, \bullet}_{g, (\beta_h,2)} \]
can be reduced to the invariants \eqref{special_case_invs}
using the relations of Propositions \ref{Proposition_WDVV_analog} and \ref{Proposition_take_out_div}.
Since, under the correspondence \eqref{eqn_correspondence},
both the genus $0$ invariants of $\Hilb^2(S)$ and the relative invariants of $S \times \p^1$ in degree $2$
are goverened by the same set of non-degenerate equations and initial values, 
they are equal.

We consider $d=1$. The invariants of $S \times \p^1$ in class $(\beta_h,1)$ with relative insertions $(1,F)$, $(1,F)$, $(1,F)$
are determined by Proposition~\ref{vanishing1} via a degeneration argument.
The result matches the corresponding series on the Hilbert scheme $\Hilb^1(S) = S$.
The remaining invariants in degree $1$ are determined by Proposition~\ref{Proposition_take_out_div}.
Hence the result follows by the same argument as above.
\qed

\subsection{The product $S \times E$} \label{Section_Proof_of_SxE_Theorem}
Let $d \geq 0$ be an integer, and let 
\[ \mathsf{N}_{g,h,d}^{S \times E} = \blangle \tau_0(F \boxtimes \omega) \brangle^{S \times E}_{g, (\beta_h,d)} \]
be the absolute reduced Gromov-Witten invariants of the product $S \times E$,
where we as usual work with the elliptically fibered K3 surface $S$ with section class $B$, fiber class $F$ and curve class $\beta_h = B + hF$.

Degenerating the elliptic curve $E$ to a nodal curve and resolving,
and degeneration off the $\tau_0(F \boxtimes \omega)$ insertion, we obtain
\begin{multline} \label{500}
\sum_{g,h} \mathsf{N}_{g,h,d}^{S \times E} u^{2g-2} q^{h-1}
= 
\sum_{\eta} \Fz(\eta) \blangle \eta, \eta^{\vee}, D(F) \brangle^{S \times \p^1, \bullet} \\
+
\chi(\Hilb^d(S)) \sum_{g,h} d! \blangle (1,\pt)^d \big| \tau_0(F \boxtimes \omega) \brangle^{S \times \p^1} u^{2g-2+2d}
q^{h-1}
\end{multline}
where $\eta$ runs over the set $\CP(d)$ of cohomology weighted partitions of size $d$
weighted by a fixed basis $\{ \gamma_i \}$,
$\eta^{\vee}$ is the dual partition of $\eta$,
and $\chi(\Hilb^d(S))$ is the topological Euler characteristic of $\Hilb^d(S)$.
The second term on the right hand side of \eqref{500} can be computed by localization
and Theorem~\ref{mainthm_1}. We obtain
\begin{multline} \label{501}
\sum_{g,h} \mathsf{N}_{g,h,d}^{S \times E} u^{2g-2} q^{h-1}
= 
\sum_{\eta} \Fz(\eta) \blangle \eta, \eta^{\vee}, D(F) \brangle^{S \times \p^1, \bullet}
+
\frac{\chi(\Hilb^d(S)) \mathbf{G}(u,q)^d}{\Theta(u,q)^2 \Delta(q)} .
\end{multline}

\begin{proof}[Proof of Theorem \ref{Theorem_K3xE}]
Under the GW/Hilb correspondence (Conjecture~\ref{GW/Hilb_correspondence}) and by a degeneration argument,
$\sum_{\eta} \Fz(\eta) \blangle \eta, \eta^{\vee}, D(F) \brangle^{S \times \p^1, \bullet}$ equals
\begin{equation} \label{123999}
\mathcal{H}_d(y,q) =
\sum_{h \geq 0} \sum_{k \in \BZ} q^{h-1} y^k
\int_{[ \Mbar_{(E,0)}(\Hilb^d(S), \beta_h + kA) ]^{\text{red}}} \ev_0^{\ast}(F)
\end{equation}
under the variable change $y = -e^{iu}$, where we follow the notation of \cite{K3xE}.
Since the GW/Hilb correspondence has been proven for $d=1$ and $d=2$ above,
the claim now follows from the Katz-Klemm-Vafa formula \cite{MPT} for $d=1$,
and Proposition~$2$ of \cite{HilbK3} for $d=2$.
Alternatively, in case $d=1$ and $d=2$ the right hand side of \eqref{501} can be directly evaluated on $S \times \p^1$
by reduction to the invariants \eqref{special_case_invs}.
\end{proof}

We analyze \eqref{501} further.
By \cite[Theorem 2]{ReducedSP} we have the expansion
\[
\sum_{g} \mathsf{N}_{g,h,d}^{S \times E} u^{2g-2} q^{h-1}
=
\sum_{g=0}^{N} \mathsf{n}_{g, h, d} (y^{1/2} + y^{-1/2})^{2g-2}
\]
where $y = -e^{iu}$ and $\mathsf{n}_{g, h, d} \in \BZ$.
A calculation of the (disconnected) genus $0$ Gromov-Witten invariants of $S \times E$
using the product formula yields
\[ \mathsf{n}_{0, h,d} = p_{24}(h) p_{24}(d), \]
where we let
\[ p_{24}(n) = \Big[ \frac{1}{\Delta(q)} \Big]_{q^{n-1}} = \chi(\Hilb^n(S)). \]
On the other hand, the coefficient of $u^{-2} q^{h-1}$ in the second term on the right hand side of \eqref{501} is
\[
\chi(\Hilb^d(S)) \Big[ \frac{\mathbf{G}(u,q)^d}{\Theta(u,q)^2 \Delta(q)} \Big]_{u^{-2} q^{h-1}}
=
p_{24}(d) \cdot p_{24}(h).
\]
This shows the following.
\begin{cor} For every $d \geq 0$ we have
\[
\sum_{g,h} \mathsf{N}_{g,h,d}^{S \times E} u^{2g-2} q^{h-1}
= 
\mathcal{F}_d(u,q)
+
\chi(\Hilb^d(S)) \frac{\mathbf{G}(u,q)^d}{\Theta(u,q)^2 \Delta(q)} \,.
\]
where under the variable change $y = e^{iu}$,
\[ \mathcal{F}_d(u,q) =
 \sum_{g=1}^{m} \mathsf{n}'_{g, h, d} (y^{1/2} + y^{-1/2})^{2g-2}.
\]
with $\mathsf{n}'_{g, h, d} \in \BZ$.
In particular $\mathcal{F}_d(u,q)$ is a holomorphic entire function in $u \in \BC$.
\end{cor}
Hence we have proven the natural splitting
of the invariants of $S \times E$ into a finite holomorphic part $\mathcal{F}_d$
(conjecturally equal to the the Hilbert scheme invariants $\mathcal{H}_d$)
and the polar part (a correction term), see the discussion of Conjecture A in \cite{K3xE}.

\appendix
\section{Gromov-Witten invariants of K3 surfaces} \label{Appendix_K3}
\subsection{Overview}
Let $S$ be an elliptic K3 surface with section, let $B$ and $F$
be the section and fiber class respectively, set
$\beta_h = B + h F$ where $h \geq 0$,
and let $\pt \in H^4(S,\BZ)$ be the class of a point.
Recall the generating series notation \eqref{Generating_Series_K3_Surfaces} for the surface $S$.
In this section we will explain how the invariants
\begin{equation} \Blangle \BE^{\vee}(1) \prod_{i} \tau_{k_i}(\pt) \prod_j \tau_{\ell_j}(F) \Brangle_g^S \label{51451451} \end{equation}
can be obtained from the Gromov-Witten theory of elliptic curves.
While the method we present yields an effective algorithm
for the computation of \eqref{51451451} for every genus $g$,
it seems difficult to obtain closed formulas in this way.
\subsection{Computation}
By degenerating $S$ to a union of $S$ with $m+n+1$-copies of $\p^1 \times E$
with each of the first $m+n$ copies receiving a marked point, and
using \eqref{KKV} for the first and \eqref{P1xE_Lemma} for the last term,
we have
\[
\Blangle \BE^{\vee}(1) \prod_{i} \tau_{k_i}(\pt) \prod_j \tau_{\ell_j}(F) \Brangle^S
=
\frac{1}{\Theta(u,q)^2 \Delta(q)} \prod_i A_{k_i}(u,q) \prod_j B_{\ell_j}(u,q)
\]
where for all $k \geq 0$ we let
\begin{alignat*}{2}
A_{k}(u,q) & = \sum_{g \geq k+1} (-1)^{g-k-1} A_{k,g}(q) u^{2g}, &
\quad \quad 
A_{k,g}(q) & = \blangle \omega \big| \lambda_{g-k-1} \tau_{k}(\pt) \big| 1 \brangle^{\p^1 \times E}_{g} \\
B_{k}(u,q) & = \sum_{g \geq k} (-1)^{g-k} B_{k,g}(q) u^{2g}, &
\quad \quad
B_{k,g}(q) & = \blangle \omega \big| \lambda_{g-k} \tau_{k}(F) \big| 1 \brangle^{\p^1 \times E}_{g}.
\end{alignat*}
By a further degeneration and Lemma \ref{P1xE_Lemma} we have
\[
A_{k,g}(q) = \blangle \omega \big| \lambda_{g-k-1} \tau_{k}(\pt) \brangle^{\p^1 \times E}_{g} \]
to which we apply the localization formula. This yields
\[
A_{k,g}(q)
 = \sum_{\substack{ i,j,\ell \geq 0 \\ 
 2i+j \leq g+\ell-1 \\
 \ell \leq g-k-1
 }} (-1)^{i+j+\ell} P(i,\ell) \cdot \blangle \tau_{g-2i-j+\ell-1}(\omega) \tau_k(\omega) \lambda_j \lambda_{g-k-1-\ell} \brangle^E_{g-i}
\]
where the invariants of a nonsingular elliptic curve $E$ are denoted by
\[
\blangle \alpha \, \tau_{k_1}(\gamma_1) \cdots \tau_{k_n}(\gamma_n) \brangle^E_g
=
\sum_{d \geq 0} \blangle \alpha \, \tau_{k_1}(\gamma_1) \cdots \tau_{k_n}(\gamma_n) \brangle^E_{g, d[E]} q^d
\]
and we set $P(0,0) = 1$, 
$P(g,\ell) = \langle \omega | \lambda_\ell \Psi_{\infty}^{g-\ell-1} | 1 \rangle^{\p^1 \times E, \sim}_{g}$
for all $g \geq \ell+1$, and $P(g,\ell) = 0$ otherwise. By the methods of \cite{MP} one proves
\[ \sum_{g,k} P(g,k) u^{2g} w^k = \exp\Big( \sum_{r \geq 1} C_{2r}(q) u^{2r} w^{r-1} \Big) \,, \]
where $C_{2r}(q)$ are the Eisenstein series \eqref{Eisenstein_Series}.
Similarly,
\[
B_{k,g}(q)
 = P(g,g-k) + \sum_{\substack{ i,j,\ell \geq 0 \\ 
 2i+j \leq g+\ell-1 \\
 \ell \leq g-k
 }} (-1)^{i+j+\ell} P(i,\ell) \cdot \blangle \tau_{g-2i-j+\ell-1}(\omega) \tau_k(1) \lambda_j \lambda_{g-k-\ell} \brangle^E_{g-i} \,.
\]
This reduces the computation of \eqref{51451451} to the evaluation of Gromov-Witten
invariants of an elliptic curve, which were completely determined
in \cite{OP3} and can be computed conveniently in the program \cite{GWall}.

We list the examples which are used in Section~\ref{Subsection_Proof_Thm_middle}.
\begin{align*}
\langle \tau_0(\pt) \rangle^{S}_{g=1} & = \frac{1}{\Delta} \big( -2 C_{2}^{2} + 10 C_{4} \big) \\
\langle \tau_1(\pt) \rangle^{S}_{g=2}           & = \frac{1}{\Delta} \big( -\frac{8}{3} C_{2}^{3} + 16 C_{2} C_{4} - 7 C_{6} \big) \\
\langle \tau_0(\pt) \lambda_1 \rangle^{S}_{g=2} & = \frac{1}{\Delta} \big( -4 C_{2}^{3} + 12 C_{2} C_{4} + 21 C_{6} \big) \\
\langle \tau_0(\pt) \tau_1(F) \rangle_{g=2}     & = \frac{1}{\Delta} \cdot 2 C_2 \cdot ( -2 C_{2}^{2} + 10 C_{4} \big).
%
\end{align*}

\end{document}